\documentclass[11pt,letterpaper,twoside,reqno,nosumlimits]{amsart}

\makeatletter{}
\usepackage{fixltx2e}                       
\usepackage[usenames,dvipsnames]{xcolor}
\usepackage{fancyhdr}
\usepackage{amsmath,amsfonts,amsbsy,amsgen,amscd,mathrsfs,amssymb,amsthm}
\usepackage{subfig}
\usepackage{url}

\usepackage{mathtools}

\mathtoolsset{showonlyrefs}

\usepackage[font=small,margin=0.25in,labelfont={sc},labelsep={colon}]{caption}

\usepackage{tikz}
\usepackage{microtype}
\usepackage{enumitem}

\definecolor{dark-gray}{gray}{0.3}
\definecolor{dkgray}{rgb}{.4,.4,.4}
\definecolor{dkblue}{rgb}{0,0,.5}
\definecolor{medblue}{rgb}{0,0,.75}
\definecolor{rust}{rgb}{0.5,0.1,0.1}

\usepackage[colorlinks=true]{hyperref}

\hypersetup{urlcolor=rust}
\hypersetup{citecolor=dkblue}
\hypersetup{linkcolor=dkblue}

\usepackage{setspace}

\usepackage{graphicx}
\usepackage{booktabs,longtable,tabu} 
\usepackage{multirow} 
\usepackage{float}

\usepackage{fourier}
\usepackage{bm}

\usepackage{bm} 
\graphicspath{{figures/}}

\newtheorem{bigthm}{Theorem}

\newtheorem{theorem}{Theorem}[section]

\newtheorem{fact}[theorem]{Fact}

\theoremstyle{definition}

\newtheorem{remark}[theorem]{Remark}

\newcommand{\term}{\emph}

\numberwithin{equation}{section} 
\numberwithin{figure}{section}
\numberwithin{table}{section}

\floatstyle{plaintop}
\newfloat{recipe}{thp}{lor}
\floatname{recipe}{Recipe}
\numberwithin{recipe}{section}

\providecommand{\mathbold}[1]{\bm{#1}}

\renewcommand{\phi}{\varphi}

\newcommand{\eps}{\varepsilon}

\newcommand{\econst}{\mathrm{e}}

\newcommand{\Id}{\mathbf{I}}

\newcommand{\coll}[1]{\mathscr{#1}}

\providecommand{\mathbbm}{\mathbb} 
\newcommand{\R}{\mathbbm{R}}
\newcommand{\C}{\mathbbm{C}}

\newcommand{\Sym}{\mathbb{H}}

\newcommand{\abs}[1]{\left\vert {#1} \right\vert}
\newcommand{\abssq}[1]{{\abs{#1}}^2}

\newcommand{\diff}[1]{\mathrm{d}{#1}}
\newcommand{\idiff}[1]{\, \diff{#1}}

\newcommand{\Prob}[1]{\mathbbm{P}\left\{{#1}\right\}}

\newcommand{\Expect}{\operatorname{\mathbb{E}}}

\newcommand{\normal}{\textsc{normal}}

\newcommand{\vct}[1]{\mathbold{#1}}
\newcommand{\mtx}[1]{\mathbold{#1}}

\newcommand{\adj}{*}

\newcommand{\trace}{\operatorname{tr}}

\newcommand{\psdle}{\preccurlyeq}

\newcommand{\norm}[1]{\left\Vert {#1} \right\Vert}
\newcommand{\normsq}[1]{\norm{#1}^2}

\newcommand{\pnorm}[2]{\norm{#2}_{#1}}

\usepackage{multicol}

\allowdisplaybreaks

\title[The Norm of a Sum of Independent Random Matrices]{The Expected Norm of a Sum of Independent Random Matrices: \\
An Elementary Approach}

\author[J.~A.~Tropp]{Joel~A.~Tropp}

\date{15 June 2015.}

\subjclass[2010]{Primary: 60B20. Secondary: 60F10, 60G50, 60G42.}
\keywords{Probability inequality; random matrix; sum of independent random variables.}

\begin{document}

\begin{abstract}
In contemporary applied and computational mathematics, a frequent
challenge is to bound the expectation of the spectral norm of a sum
of independent random matrices.  This quantity is controlled by the
norm of the expected square of the random matrix and the expectation
of the maximum squared norm achieved by one of the summands;
there is also a weak dependence on the dimension of the random matrix.
The purpose of this paper is to give a complete, elementary proof
of this important, but underappreciated, inequality.
\end{abstract}

\maketitle

\section{Motivation}

Over the last decade, random matrices have become ubiquitous in applied and computational mathematics.
As this trend accelerates, more and more researchers must confront random matrices as part of their work.
Classical random matrix theory can be difficult to use, and it is often silent about the questions that come
up in modern applications.  As a consequence, it has become imperative to develop and disseminate
new tools that are easy to use and that apply to a wide range of random matrices.

\subsection{Matrix Concentration Inequalities}

Matrix concentration inequalities are among the most popular of these new methods.
For a random matrix $\mtx{Z}$ with appropriate structure, these results use simple
parameters associated with the random matrix to provide bounds of the form
$$
\Expect \norm{ \mtx{Z} - \Expect \mtx{Z} } \quad \leq \quad \dots
\quad\text{and}\quad
\Prob{ \norm{ \mtx{Z} - \Expect \mtx{Z} } \geq t }
	\quad\leq\quad \dots
$$
where $\norm{\cdot}$ denotes the spectral norm, also known as the $\ell_2$ operator norm.
These tools have already found a place in a huge number of mathematical research fields, including
\begin{multicols}{2}
\begin{itemize}
\item	numerical linear algebra~\cite{Tro11:Improved-Analysis}
\item	numerical analysis~\cite{MB14:Far-Field-Compression}
\item	uncertainty quantification~\cite{CG14:Computing-Active}
\item	statistics~\cite{Kol11:Oracle-Inequalities}
\item	econometrics~\cite{CC13:Optimal-Uniform}
\item	approximation theory~\cite{CDL13:Stability-Accuracy}
 \item	sampling theory~\cite{BG13:Relevant-Sampling}
\item	machine learning~\cite{DKC13:High-Dimensional-Gaussian,LSS+14:Randomized-Nonlinear}
\item	learning theory~\cite{FSV12:Learning-Functions,MKR12:PAC-Bayesian}
\item	mathematical signal processing~\cite{CBSW14:Coherent-Matrix}
\item	optimization~\cite{CSW12:Linear-Matrix}
\item	computer graphics and vision~\cite{CGH14:Near-Optimal-Joint}
\item	quantum information theory~\cite{Hol12:Quantum-Systems}
\item	theory of algorithms~\cite{HO14:Pipage-Rounding,CKMP14:Solving-SDD} and
\item	combinatorics~\cite{Oli10:Spectrum-Random}.
\end{itemize}
\end{multicols}
\noindent
These references are chosen more or less at random from a long menu of possibilities.
See the monograph~\cite{Tro15:Introduction-Matrix}
for an overview of the main results on matrix concentration,
many detailed applications, and additional background references.

\subsection{The Expected Norm}

The purpose of this paper is to provide a complete proof of the following
important, but underappreciated, theorem.
This result is adapted from~\cite[Thm.~A.1]{CGT12:Masked-Sample}.

\begin{bigthm}[The Expected Norm of an Independent Sum of Random Matrices] \label{thm:main}
Consider an independent family $\{ \mtx{S}_1, \dots, \mtx{S}_n \}$ of random $d_1 \times d_2$
complex-valued matrices with $\Expect \mtx{S}_i = \mtx{0}$ for each index $i$, and define 
\begin{equation} \label{eqn:indep-sum}
\mtx{Z} := \sum_{i=1}^n \mtx{S}_i.
\end{equation}
Introduce the matrix variance parameter
\begin{equation} \label{eqn:variance-param}
\begin{aligned}
v(\mtx{Z}) :=& \max\left\{ \norm{ \Expect\big[ \mtx{ZZ}^\adj \big] }, \
	\norm{ \Expect\big[ \mtx{Z}^\adj \mtx{Z} \big] } \right\} \\
	=& \max\left\{ \norm{ \sum_i \Expect\big[ \mtx{S}_i \mtx{S}_i^\adj \big] }, \
	\norm{ \sum_i \Expect\big[ \mtx{S}_i^\adj \mtx{S}_i \big] } \right\}
\end{aligned}
\end{equation}
and the large deviation parameter
\begin{equation} \label{eqn:large-dev-param}
L := \left( \Expect \max\nolimits_i \normsq{\mtx{S}_i} \right)^{1/2}.
\end{equation}
Define the dimensional constant
\begin{equation} \label{eqn:dimensional}
C(\vct{d}) := C(d_1, d_2) := 4 \cdot \big(1 + 2\lceil \log (d_1 + d_2) \rceil \big).
\end{equation}
Then we have the matching estimates
\begin{equation} \label{eqn:main-ineqs}
\sqrt{c \cdot v(\mtx{Z})} \ +\ c \cdot L 
	\quad\leq\quad \left( \Expect \normsq{\mtx{Z}} \right)^{1/2}
	\quad\leq\quad \sqrt{C(\vct{d}) \cdot v(\mtx{Z})}\ +\ C(\vct{d}) \cdot L.
\end{equation}
In the lower inequality, we can take $c := 1/4$.
The symbol $\norm{\cdot}$ denotes the $\ell_2$ operator norm, also known as the spectral norm,
and ${}^*$ refers to the conjugate transpose operation.  The map $\lceil \cdot \rceil$ returns
the smallest integer that exceeds its argument.
\end{bigthm}

\noindent
The proof of this result occupies the bulk of this paper.
Most of the page count is attributed to a detailed
presentation of the required background material from linear
algebra and probability.  We have based the argument on the
most elementary considerations possible, and we have tried
to make the work self-contained.
Once the reader has digested these ideas, the related---but more sophisticated
---approach in the paper~\cite{MJCFT14:Matrix-Concentration} should be accessible.

\subsection{Discussion}

Before we continue, some remarks about Theorem~\ref{thm:main} are in order.
First, although it may seem restrictive to focus on
independent sums, as in~\eqref{eqn:indep-sum},
this model captures an enormous number of
useful examples.  See the monograph~\cite{Tro15:Introduction-Matrix}
for justification.

We have chosen the term \emph{variance parameter} because the quantity~\eqref{eqn:variance-param}
is a direct generalization of the variance of a scalar random variable.
The passage from the first formula to the second formula in~\eqref{eqn:variance-param}
is an immediate consequence of the assumption that the summands $\mtx{S}_i$ are independent
and have zero mean (see Section~\ref{sec:upper}).
We use the term \emph{large-deviation parameter} because the
quantity~\eqref{eqn:large-dev-param} reflects the part of the expected
norm of the random matrix that is attributable to
one of the summands taking an unusually large value.  In practice,
both parameters are easy to compute using matrix arithmetic and
some basic probabilistic considerations.

In applications, it is common that we need high-probability bounds on
the norm of a random matrix.  Typically, the bigger challenge is to
estimate the expectation of the norm, which is what Theorem~\ref{thm:main}
achieves.  Once we have a bound for the expectation, we can use scalar
concentration inequalities, such as~\cite[Thm.~6.10]{BLM13:Concentration-Inequalities},
to obtain high-probability bounds on the deviation between the norm and
its mean value.

We have stated Theorem~\ref{thm:main} as a bound on the second moment of $\norm{\mtx{Z}}$
because this is the most natural form of the result.  Equivalent bounds hold for the first
moment:
$$
\sqrt{c' \cdot v(\mtx{Z})} \ +\  c' \cdot L
	\quad\leq\quad \Expect \norm{ \mtx{Z} }
	\quad\leq\quad \sqrt{C(\vct{d}) \cdot v(\mtx{Z})}\ +\ C(\vct{d}) \cdot L.
$$
We can take $c' = 1/8$.
The upper bound follows easily from~\eqref{eqn:main-ineqs} and Jensen's inequality.
The lower bound requires the Khintchine--Kahane inequality~\cite{LO94:Best-Constant}.

Observe that the lower and upper estimates in~\eqref{eqn:main-ineqs}
differ only by the factor $C(\vct{d})$.  As a consequence,
the lower bound has no explicit dimensional dependence,
while the upper bound has only a weak dependence on the dimension.
Under the assumptions of the theorem, it is not possible to make
substantial improvements to either the lower bound or the
upper bound.  Section~\ref{sec:examples} provides examples
that support this claim.

In the theory of matrix concentration, one of the major challenges
is to understand what properties of the random matrix $\mtx{Z}$
allow us to remove the dimensional factor $C(\vct{d})$ from the
estimate~\eqref{eqn:main-ineqs}.  This question is largely open,
but the recent papers~\cite{Oli13:Lower-Tail,BV14:Sharp-Nonasymptotic,Tro15:Second-Order-Matrix}
make some progress.

\subsection{The Uncentered Case}

Although Theorem~\ref{thm:main} concerns a centered random matrix,
it can also be used to study a general random matrix.  The following
result is an immediate corollary of Theorem~\ref{thm:main}.

\begin{bigthm} \label{thm:uncentered}
Consider an independent family $\{ \mtx{S}_1, \dots, \mtx{S}_n \}$ of random $d_1 \times d_2$
complex-valued matrices, not necessarily centered.  Define
$$
\mtx{R} := \sum_{i=1}^n \mtx{S}_i
$$
Introduce the variance parameter
$$
\begin{aligned}
v(\mtx{R}) :=& \max\left\{ \norm{ \Expect \big[ (\mtx{R} - \Expect \mtx{R})(\mtx{R} - \Expect \mtx{R})^\adj \big] }, \
	\norm{ \Expect \big[ (\mtx{R} - \Expect \mtx{R})^\adj (\mtx{R} - \Expect \mtx{R}) \big] } \right\} \\
	=& \max\left\{ \norm{ \sum_{i=1}^n \Expect \big[ (\mtx{S}_i - \Expect \mtx{S}_i)(\mtx{S}_i - \Expect \mtx{S}_i)^\adj \big] }, \
	\norm{ \Expect \big[ (\mtx{S}_i - \Expect \mtx{S}_i)^\adj (\mtx{S}_i - \Expect \mtx{S}_i) \big] } \right\}
\end{aligned}
$$
and the large-deviation parameter
$$
L^2 := \Expect \max\nolimits_i \normsq{ \mtx{S}_i - \Expect \mtx{S}_i }.
$$
Then we have the matching estimates
$$
\sqrt{c \cdot v(\mtx{R})}\ +\ c \cdot L
	\quad\leq\quad \left( \Expect \normsq{\mtx{R} - \Expect \mtx{R}} \right)^{1/2}
	\quad\leq\quad \sqrt{C(\vct{d}) \cdot v(\mtx{R})} \ + \ C(\vct{d}) \cdot L.
$$
We can take $c = 1/4$, and the dimensional constant $C(\vct{d})$ is defined in~\eqref{eqn:dimensional}.
\end{bigthm}

Theorem~\ref{thm:uncentered} can also be used to study $\norm{\mtx{R}}$
by combining it with the estimates
$$
\norm{ \Expect \mtx{R} } \ -\ \left( \Expect \normsq{ \mtx{R} - \Expect \mtx{R} } \right)^{1/2}
\quad\leq\quad \left( \Expect \normsq{\mtx{R}} \right)^{1/2}
\quad\leq\quad \norm{ \Expect \mtx{R} } \ + \ \left( \Expect \normsq{ \mtx{R} - \Expect \mtx{R} } \right)^{1/2}.
$$
These bounds follow from the triangle inequality for the spectral norm.

It is productive to interpret Theorem~\ref{thm:uncentered} as a perturbation result
because it describes how far the random matrix $\mtx{R}$ deviates from its mean $\Expect \mtx{R}$.
We can derive many useful consequences from a bound of the form
$$
\left( \Expect \normsq{ \mtx{R} - \Expect \mtx{R} } \right)^{1/2}
	\quad\leq\quad \dots
$$
This estimate shows that, on average, all of the singular values of $\mtx{R}$
are close to the corresponding singular values of $\Expect \mtx{R}$.
It also implies that, on average, the singular vectors of $\mtx{R}$ are close to the
corresponding singular vectors of $\Expect \mtx{R}$, provided that the
associated singular values are isolated.  Furthermore, we discover that, on average, each
linear functional $\trace[ \mtx{CR} ]$ is uniformly close to $\Expect \trace[ \mtx{CR} ]$
for each fixed matrix $\mtx{C} \in \mathbb{M}^{d_2 \times d_1}$ with bounded
Schatten $1$-norm $\pnorm{S_1}{\mtx{C}} \leq 1$.

\subsection{History}

Theorem~\ref{thm:main} is not new.  A somewhat weaker version of the upper bound
appeared in Rudelson's work~\cite[Thm.~1]{Rud99:Random-Vectors};
see also~\cite[Thm.~3.1]{RV07:Sampling-Large} and~\cite[Sec.~9]{Tro08:Conditioning-Random}.
The first explicit statement of the upper bound appeared
in~\cite[Thm.~A.1]{CGT12:Masked-Sample}.  All of these results
depend on the noncommutative Khintchine
inequality~\cite{LP86:Inegalites-Khintchine,Pis98:Noncommutative-Vector,Buc01:Operator-Khintchine}.
In our approach, the main innovation is a particularly easy proof of a
Khintchine-type inequality for matrices,
patterned after~\cite[Cor~7.3]{MJCFT14:Matrix-Concentration}
and~\cite[Thm.~8.1]{Tro15:Second-Order-Matrix}.

The ideas behind the proof of the lower bound in Theorem~\ref{thm:main} are older.
This estimate depends on generic considerations about the behavior
of a sum of independent random variables in a Banach space.
These techniques are explained in detail in~\cite[Ch.~6]{LT91:Probability-Banach}.
Our presentation expands on a proof sketch that appears in
the monograph~\cite[Secs.~5.1.2 and~6.1.2]{Tro15:Introduction-Matrix}.

\subsection{Target Audience}

This paper is intended for students and researchers who want to
develop a detailed understanding of the foundations of matrix concentration.
The preparation required is modest.

\begin{itemize}

\item	\textbf{Basic Convexity.}  Some simple ideas from convexity play a role,
notably the concept of a convex function and Jensen's inequality.

\item	\textbf{Intermediate Linear Algebra.}  The requirements from linear algebra are more substantial.  The reader should be familiar with the spectral theorem for Hermitian (or symmetric) matrices, Rayleigh's variational principle, the trace of a matrix, and the spectral norm.  The paper includes reminders about this material.  The paper elaborates on some less familiar ideas, including inequalities for the trace and the spectral norm.

\item	\textbf{Intermediate Probability.}  The paper demands some comfort with probability.  The most important concepts are expectation and the elementary theory of conditional expectation.  We develop the other key ideas, including the notion of symmetrization.
\end{itemize}

\noindent
Although many readers will find the background material unnecessary,
it is hard to locate these ideas in one place and we prefer to make the
paper self-contained.  In any case, we provide detailed cross-references
so that the reader may dive into the proofs of the main results
without wading through the shallower part of the paper.

\subsection{Roadmap}

Section~\ref{sec:linear-algebra} and Section~\ref{sec:probability} contain
the background material from linear algebra and probability.
To prove the upper bound in Theorem~\ref{thm:main}, the
key step is to establish the result for the special case
of a sum of fixed matrices, each modulated by a random sign.
This result appears in Section~\ref{sec:khintchine}.
In Section~\ref{sec:upper}, we exploit this result
to obtain the upper bound in~\eqref{eqn:main-ineqs}.
In Section~\ref{sec:lower}, we present the easier proof
of the lower bound in~\eqref{eqn:main-ineqs}.
Finally, Section~\ref{sec:examples} shows that
it is not possible to improve~\eqref{eqn:main-ineqs}
substantially.

\section{Linear Algebra Background} \label{sec:linear-algebra}

Our aim is to make this paper as accessible as possible.
To that end, this section presents some background material from linear algebra.
Good references include~\cite{Hal74:Finite-Dimensional-Vector,Bha97:Matrix-Analysis,HJ13:Matrix-Analysis}.
We also assume some familiarity with basic ideas from the theory of convexity,
which may be found in the
books~\cite{Lue69:Optimization-Vector,Roc70:Convex-Analysis,Bar02:Course-Convexity,BV04:Convex-Optimization}.

\subsection{Convexity}

Let $V$ be a finite-dimensional linear space.  A subset $E \subset V$ is \term{convex} when
$$
\vct{x}, \vct{y} \in E
\quad\text{implies}\quad
\tau \cdot \vct{x} + (1- \tau) \cdot \vct{y} \in E
\quad\text{for each $\tau \in [0, 1]$.}
$$
Let $E$ be a convex subset of a linear space $V$.  A function $f : E \to \R$
is \term{convex} if
\begin{equation} \label{eqn:convexity}
f\big( \tau \vct{x} + (1-\tau) \vct{y} \big) \leq \tau \cdot f(\vct{x}) + (1-\tau) \cdot f(\vct{y})
\quad\text{for all $\tau \in [0,1]$ and all $\vct{x}, \vct{y} \in V$.}
\end{equation}
We say that $f$ is \term{concave} when $-f$ is convex.

\subsection{Vector Basics}

Let $\C^d$ be the complex linear space of $d$-dimensional complex vectors,
equipped with the usual componentwise addition and scalar multiplication.
The $\ell_2$ norm $\norm{\cdot}$ is defined on $\C^d$ via the expression
\begin{equation*} \label{eqn:l2-norm}
\normsq{ \vct{x} } := \vct{x}^\adj \vct{x}
\quad\text{for each $\vct{x} \in \C^d$.}
\end{equation*}
The symbol ${}^\adj$ denotes the conjugate transpose of a vector.
Recall that the $\ell_2$ norm is a convex function.

A family $\{ \vct{u}_1, \dots, \vct{u}_d \} \subset \C^d$
is called an \term{orthonormal basis} if it satisfies the relations
$$
\vct{u}_i^\adj \vct{u}_j = \begin{cases} 1, & i = j \\ 0, & i \neq j. \end{cases}
$$
The orthonormal basis also has the property
$$
\sum_{i=1}^d \vct{u}_i \vct{u}_i^\adj = \Id_d
$$
where $\Id_d$ is the $d \times d$ identity matrix.

\subsection{Matrix Basics}

A \term{matrix} is a rectangular array of complex numbers.
Addition and multiplication by a complex scalar are defined componentwise,
and we can multiply two matrices with compatible dimensions. 
We write $\mathbb{M}^{d_1 \times d_2}$
for the complex linear space of $d_1 \times d_2$ matrices.
The symbol ${}^\adj$ also refers to the conjugate transpose operation on matrices.

A square matrix $\mtx{H}$ is \term{Hermitian} when $\mtx{H} = \mtx{H}^\adj$.
Hermitian matrices are sometimes called \term{conjugate symmetric}.
We introduce the set of $d \times d$ Hermitian matrices:
$$
\Sym_d := \big\{ \mtx{H} \in \mathbb{M}^{d \times d} : \mtx{H} = \mtx{H}^{\adj} \big\}.
$$
Note that the set $\Sym_d$ is a linear space over the real field.

An Hermitian matrix $\mtx{A} \in \Sym_d$ is \term{positive semidefinite} when
\begin{equation*} \label{eqn:psd}
\vct{u}^\adj \mtx{A} \vct{u} \geq 0
\quad\text{for each $\vct{u} \in \C^d$.}
\end{equation*}
It is convenient to use the notation $\mtx{A} \psdle \mtx{H}$ to mean that $\mtx{H} - \mtx{A}$
is positive semidefinite.  In particular, the relation $\mtx{0} \psdle \mtx{H}$ is equivalent
to $\mtx{H}$ being positive semidefinite.  Observe that
$$
\mtx{0} \psdle \mtx{A}
\quad\text{and}\quad
\mtx{0} \psdle \mtx{H}
\quad\text{implies}\quad
\mtx{0} \psdle \alpha \cdot (\mtx{A} + \mtx{H})
\quad\text{for each $\alpha \geq 0$.}
$$
In other words, addition and nonnegative scaling preserve the positive-semidefinite property.

For every matrix $\mtx{B}$, both of its squares $\mtx{BB}^\adj$ and $\mtx{B}^\adj \mtx{B}$
are Hermitian and positive semidefinite.

\subsection{Basic Spectral Theory}

Each Hermitian matrix $\mtx{H} \in \Sym_d$ can be expressed in the form
\begin{equation} \label{eqn:eig-decomp}
\mtx{H} = \sum_{i=1}^d \lambda_i \vct{u}_i \vct{u}_i^\adj
\end{equation}
where the $\lambda_i$ are uniquely determined real numbers, called \term{eigenvalues},
and $\{ \vct{u}_i \}$ is an orthonormal basis for $\C^d$.
The representation~\eqref{eqn:eig-decomp} is called an \term{eigenvalue decomposition}.

An Hermitian matrix $\mtx{H}$ is positive semidefinite if and only if its eigenvalues $\lambda_i$ are all nonnegative.
Indeed, using the eigenvalue decomposition~\eqref{eqn:eig-decomp}, we see that
$$
\vct{u}^\adj \mtx{H} \vct{u}
	= \sum_{i=1}^n \lambda_i \cdot \vct{u}^\adj \vct{u}_i \vct{u}_i^\adj \vct{u}
	= \sum_{i=1}^n \lambda_i \cdot \abssq{ \smash{\vct{u}^\adj \vct{u}_i} }.
$$
To verify the forward direction, select $\vct{u} = \vct{u}_j$ for each index $j$.
The reverse direction should be obvious.

We define a \term{monomial} function of an Hermitian matrix $\mtx{H} \in \Sym_d$ by repeated multiplication:
$$
\mtx{H}^0 = \Id_d,
\quad \mtx{H}^1 = \mtx{H},
\quad \mtx{H}^2 = \mtx{H} \cdot \mtx{H},
\quad \mtx{H}^3 = \mtx{H} \cdot \mtx{H}^2,
\quad\text{etc.}
$$
For each nonnegative integer $r$, it is not hard to check that
\begin{equation} \label{eqn:monomial}
\mtx{H} = \sum_{i=1}^d \lambda_i \vct{u}_i \vct{u}_i^\adj
\quad\text{implies}\quad
\mtx{H}^r = \sum_{i=1}^d \lambda_i^r \vct{u}_i \vct{u}_i^\adj.
\end{equation}
In particular, $\mtx{H}^{2p}$ is positive semidefinite for each nonnegative integer $p$.

\subsection{Rayleigh's Variational Principle}

The \term{Rayleigh principle} is an attractive expression
for the maximum eigenvalue $\lambda_{\max}(\mtx{H})$
of an Hermitian matrix $\mtx{H} \in \Sym_d$.  This result states that
\begin{equation} \label{eqn:rayleigh}
\lambda_{\max}(\mtx{H}) = \max_{\norm{ \vct{u}} =1 } \ \vct{u}^\adj \mtx{H} \vct{u}.
\end{equation}
The maximum takes place over all unit-norm vectors $\vct{u} \in \C^d$.  The identity~\eqref{eqn:rayleigh}
follows from the Lagrange multiplier theorem and the existence of the eigenvalue decomposition~\eqref{eqn:eig-decomp}.
Similarly, the minimum eigenvalue $\lambda_{\min}(\mtx{H})$ satisfies
\begin{equation} \label{eqn:rayleigh-min}
\lambda_{\min}(\mtx{H}) = \min_{\norm{ \vct{u}} = 1 } \ \vct{u}^\adj \mtx{H} \vct{u}.
\end{equation}
We can obtain~\eqref{eqn:rayleigh-min} by applying~\eqref{eqn:rayleigh} to $-\mtx{H}$.

Rayleigh's principle implies that order relations for positive-semidefinite matrices
lead to order relations for their eigenvalues.

\begin{fact}[Monotonicity] \label{fact:weyl}
Let $\mtx{A}, \mtx{H} \in \Sym_d$ be Hermitian matrices.  Then 
\begin{equation*} \label{eqn:weyl}
\mtx{A} \psdle \mtx{H}
\quad\text{implies}\quad
\lambda_{\max}(\mtx{A}) \leq \lambda_{\max}(\mtx{H}).
\end{equation*}
\end{fact}

\begin{proof}
The condition $\mtx{A} \psdle \mtx{H}$ implies that the eigenvalues of
$\mtx{H} - \mtx{A}$ are nonnegative.  Therefore, Rayleigh's principle~\eqref{eqn:rayleigh-min}
yields
$$
0 \leq \lambda_{\min}(\mtx{H} - \mtx{A})
	= \min_{\norm{\vct{u}} = 1} \left( \vct{u}^\adj \mtx{H} \vct{u} - \vct{u}^\adj \mtx{A} \vct{u} \right)
	\leq \vct{v}^\adj \mtx{H} \vct{v} - \vct{v}^\adj \mtx{A} \vct{v}
$$
for any unit-norm vector $\vct{v}$.  Select a unit-norm vector $\vct{v}$
for which $\lambda_{\max}(\mtx{A}) = \vct{v}^\adj \mtx{A} \vct{v}$,
and then rearrange:
$$
\lambda_{\max}( \mtx{A} ) = \vct{v}^\adj \mtx{A} \vct{v} \leq \vct{v}^\adj \mtx{H} \vct{v}
	\leq \lambda_{\max}(\mtx{H}).
$$
The last relation is Rayleigh's principle~\eqref{eqn:rayleigh}.
\end{proof}

\subsection{The Trace}

The \term{trace} of a square matrix $\mtx{B} \in \mathbb{M}^{d \times d}$ is defined as
\begin{equation} \label{eqn:trace}
\trace \mtx{B} := \sum_{i=1}^d b_{ii}.
\end{equation}
It is clear that the trace is a linear functional on $\mathbb{M}^{d \times d}$.
By direct calculation, one may verify that
$$
\trace[ \mtx{BC} ] = \trace[ \mtx{CB} ]
\quad\text{for all $\mtx{B} \in \mathbb{M}^{d \times r}$ and $\mtx{C} \in \mathbb{M}^{r \times d}$.}
$$
This property is called the \term{cyclicity} of the trace.

The trace of an Hermitian matrix $\mtx{H} \in \Sym_d$ can also be expressed in terms of its eigenvalues:
\begin{equation} \label{eqn:trace-eig}
\trace \mtx{H} = \sum_{i=1}^d \lambda_i.
\end{equation}
This formula follows when we introduce the eigenvalue decomposition~\eqref{eqn:eig-decomp}
into~\eqref{eqn:trace}.
Then we invoke the linearity and the cyclicity properties of the trace,
as well as the properties of an orthonormal basis. We also instate the convention
that monomials bind before the trace: $\trace \mtx{H}^r := \trace[ \mtx{H}^r ]$
for each nonnegative integer $r$.

\subsection{The Spectral Norm}

The \term{spectral norm} of a matrix $\mtx{B} \in \mathbb{M}^{d_1 \times d_2}$ is defined as
\begin{equation*} \label{eqn:spectral-norm}
\norm{ \mtx{B} } := \max_{\norm{\vct{u}} = 1}\ \norm{ \mtx{B} \vct{u} }.
\end{equation*}
The maximum takes place over unit-norm vectors $\vct{u} \in \C^{d_2}$.
We have the important identity
\begin{equation} \label{eqn:B*}
\normsq{ \mtx{B} } = \norm{ \smash{\mtx{B}^\adj \mtx{B}} } = \norm{ \smash{\mtx{BB}^\adj} }
\quad\text{for every matrix $\mtx{B}$.}
\end{equation}
Furthermore,
the spectral norm is a convex function, and it satisfies the triangle inequality.

For an Hermitian matrix, the spectral norm can be written in terms of the eigenvalues:
\begin{equation} \label{eqn:norm-herm}
\norm{ \mtx{H} } = \max\big\{ \lambda_{\max}(\mtx{H}), \ - \lambda_{\min}(\mtx{H}) \big\}
\quad\text{for each Hermitian matrix $\mtx{H}$.}
\end{equation}
As a consequence,
\begin{equation} \label{eqn:norm-psd}
\norm{ \mtx{A} } = \lambda_{\max}(\mtx{A})
\quad\text{for each positive-semidefinite matrix $\mtx{A}$.}
\end{equation}
This discussion implies that
\begin{equation} \label{eqn:norm-power}
\norm{ \mtx{H} }^{2p} = \norm{ \smash{\mtx{H}^{2p}} }
\quad\text{for each Hermitian $\mtx{H}$ and each nonnegative integer $p$.}
\end{equation}
Use the relations~\eqref{eqn:monomial} and~\eqref{eqn:norm-herm}
to verify this fact.

\subsection{Some Spectral Norm Inequalities}

We need some basic inequalities for the spectral norm.  First,
note that
\begin{equation} \label{eqn:norm-trace-bd}
\norm{ \mtx{A} } \leq \trace \mtx{A}
\quad\text{when $\mtx{A}$ is positive semidefinite.}
\end{equation}
This point follows from~\eqref{eqn:norm-psd}
and~\eqref{eqn:trace-eig} because the eigenvalues
of a positive-semidefinite matrix are nonnegative.

The next result uses the spectral norm to bound the trace of a product.

\begin{fact}[Bound for the Trace of a Product] \label{fact:trace-dual}
Consider Hermitian matrices $\mtx{A}, \mtx{H} \in \Sym_d$, and assume that $\mtx{A}$ is positive semidefinite.  Then
\begin{equation*} \label{eqn:trace-dual}
\trace[ \mtx{HA} ] \leq \norm{\mtx{H}} \cdot \trace \mtx{A}.
\end{equation*}
\end{fact}

\begin{proof}
Introducing the eigenvalue decomposition $\mtx{A} = \sum_i \lambda_i \vct{u}_i \vct{u}_i^\adj$, we see that
$$
\trace[ \mtx{HA} ] = \sum_i \lambda_i \trace[ \mtx{H} \vct{u}_i \vct{u}_i^\adj ]
	= \sum_i \lambda_i \vct{u}_i^\adj \mtx{H} \vct{u}_i
	\leq \lambda_{\max}(\mtx{H}) \cdot \sum_i \lambda_i 
	\leq \norm{ \mtx{H} } \cdot \trace \mtx{A}.
$$
The first two relations follow from linearity and cyclicity of the trace.  The first inequality
depends on Rayleigh's principle~\eqref{eqn:rayleigh} and the nonnegativity of the
eigenvalues $\lambda_i$.  The last bound follows from~\eqref{eqn:norm-herm}.
\end{proof}

We also need a bound for the norm of a sum of squared positive-semidefinite matrices.

\begin{fact}[Bound for a Sum of Squares] \label{fact:sum-squares}
Consider positive-semidefinite matrices $\mtx{A}_1, \dots, \mtx{A}_n \in \Sym_d$.  Then
$$
\norm{ \sum_{i=1}^n \mtx{A}_i^2 } \leq
\max\nolimits_i \norm{ \mtx{A}_i } \cdot \norm{ \sum_{i=1}^n \mtx{A}_i }.
$$
\end{fact}
\begin{proof}
Let $\mtx{A}$ be positive semidefinite.
We claim that
\begin{equation} \label{eqn:square-claim}
\mtx{A}^2 \psdle M \cdot \mtx{A}
\quad\text{whenever $\lambda_{\max}(\mtx{A}) \leq M$.}
\end{equation}
Indeed, introducing the eigenvalue decomposition $\mtx{A} = \sum_i \lambda_i \vct{u}_i \vct{u}_i^\adj$,
we find that
$$
M \cdot \mtx{A} - \mtx{A}^2
	= M \cdot \sum_i \lambda_i \vct{u}_i \vct{u}_i^\adj
		- \sum_i \lambda_i^2 \vct{u}_i\vct{u}_i^\adj
	= \sum_i \big(M - \lambda_i \big) \lambda_i \cdot  \vct{u}_i \vct{u}_i^\adj.
$$
The first relation uses~\eqref{eqn:monomial}.  Since $0 \leq \lambda_i \leq M$,
the scalar coefficients in the sum are nonnegative.  Therefore, the matrix
$M \cdot \mtx{A} - \mtx{A}^2$ is positive semidefinite, which
is what we needed to show.

Select $M := \max_i \lambda_{\max}(\mtx{A}_i)$.  The inequality~\eqref{eqn:square-claim} ensures that
$$
\mtx{A}_i^2 \psdle M \cdot \mtx{A}_i
\quad\text{for each index $i$.}
$$
Summing these relations, we see that
$$
\sum_{i=1}^n \mtx{A}_i^2 \psdle M \cdot \sum_{i=1}^n \mtx{A}_i.
$$
The monotonicity principle, Fact~\ref{fact:weyl}, yields the inequality
$$
\lambda_{\max}\left( \sum_{i=1}^n \mtx{A}_i^2 \right)
	\leq \lambda_{\max}\left( M \cdot \sum_{i=1}^n \mtx{A}_i \right)
	= M \cdot \lambda_{\max}\left( \sum_{i=1}^n \mtx{A}_i \right).
$$
We have used the fact that the maximum eigenvalue of an Hermitian matrix is positive homogeneous.
Finally, recall that, per~\eqref{eqn:norm-psd},
the spectral norm of a positive-semidefinite matrix equals its maximum eigenvalue.
\end{proof}	

\subsection{GM--AM Inequality for the Trace}

We require another substantial matrix inequality, which is one
(of several) matrix analogs of the inequality between the geometric mean
and the arithmetic mean. 
\begin{fact}[GM--AM Trace Inequality] \label{fact:trace-gm-am}
Consider Hermitian matrices $\mtx{H}, \mtx{W}, \mtx{Y} \in \Sym_d$.
For each nonnegative integer $r$ and each integer $q$ in the range $0 \leq q \leq 2r$,
\begin{equation} \label{eqn:trace-gm-am}
\trace\big[ \mtx{H}\mtx{W}^q \mtx{H} \mtx{Y}^{2r-q} \big]
+ \trace\big[ \mtx{H}\mtx{W}^{2r-q} \mtx{H} \mtx{Y}^{q} \big]
	\leq \trace\big[ \mtx{H}^2 \cdot \big( \mtx{W}^{2r} + \mtx{Y}^{2r} \big) \big].
\end{equation}
In particular,
\begin{equation*} \sum_{q=0}^{2r} \trace\big[ \mtx{H}\mtx{W}^q \mtx{H} \mtx{Y}^{2r-q} \big]
	\leq \frac{2r+1}{2} \trace\big[ \mtx{H}^2 \cdot \big( \mtx{W}^{2r} + \mtx{Y}^{2r} \big) \big].
\end{equation*}
\end{fact}

\begin{proof}
The result~\eqref{eqn:trace-gm-am} is a matrix version of the following numerical inequality.
For $\lambda, \mu \geq 0$,
\begin{equation} \label{eqn:heinz}
\lambda^{\theta} \mu^{1 - \theta} + \lambda^{1-\theta} \mu^{\theta}
	\leq \lambda + \mu
	\quad\text{for each $\theta \in [0,1]$.}
\end{equation}
To verify this bound, we may assume that $\lambda, \mu > 0$ because
it is trivial to check when either $\lambda$ or $\mu$ equals zero.
Notice that the left-hand side of the bound is a
convex function of $\theta$ on the interval $[0,1]$.  This point follows easily from
the representation
$$
f(\theta) := \lambda^{\theta} \mu^{1 - \theta} + \lambda^{1-\theta} \mu^{\theta}
	= \econst^{\theta \log \lambda + (1-\theta) \log \mu}
	+ \econst^{(1-\theta) \log \lambda + \theta \log \mu}.
$$
The value of the convex function $f$ on the interval $[0,1]$ is controlled by
the maximum value it achieves at one of the endpoints:
$$
f(\theta) \leq \max\big\{ f(0), f(1) \big\}
	= \lambda + \mu.
$$
This inequality coincides with~\eqref{eqn:heinz}.

To prove~\eqref{eqn:trace-gm-am} from~\eqref{eqn:heinz}, we use eigenvalue decompositions.
The case $r = 0$ is immediate, so we may assume that $r \geq 1$.
Let $q$ be an integer in the range $0 \leq q \leq 2r$.  Introduce
eigenvalue decompositions:
$$
\mtx{W} = \sum_{i=1}^d \lambda_i \vct{u}_i \vct{u}_i^\adj
\quad\text{and}\quad
\mtx{Y} = \sum_{j=1}^d \mu_j \vct{v}_j \vct{v}_j^\adj.
$$
Calculate that
\begin{equation} \label{eqn:gm-am-temp}
\begin{aligned}
\trace \big[ \mtx{H}\mtx{W}^q \mtx{H} \mtx{Y}^{2r-q} \big]
	&= \trace \left[ \mtx{H} \big(\sum_{i=1}^d \lambda_i^q \vct{u}_i\vct{u}_i^\adj \big)
	\mtx{H} \big( \sum_{j=1}^d \mu_j^{2r-q} \vct{v}_j \vct{v}_j^\adj \big) \right] \\
	&=\sum_{i,j=1}^d \lambda_i^q \mu_j^{2r-q} \cdot
	\trace \big[ \mtx{H} \vct{u}_i \vct{u}_i^\adj \mtx{H} \vct{v}_j \vct{v}_j^\adj \big] \\
	&\leq \sum_{i,j=1}^d \abs{\lambda_i}^q \abs{\smash{\mu_j}}^{2r-q} \cdot\abssq{ \vct{u}_i^\adj \mtx{H} \vct{v}_j }.
\end{aligned}
\end{equation}
The first identity relies on the formula~\eqref{eqn:monomial} for the eigenvalue decomposition of a monomial.
The second step depends on the linearity of the trace.
In the last line, we rewrite the trace using cyclicity,
and the inequality emerges when we apply absolute values.
The representation $\abssq{ \smash{\vct{u}_i^\adj \mtx{H} \vct{v}_j} }$
emphasizes that this quantity is nonnegative, which we use
to justify several inequalities.

Invoking the inequality~\eqref{eqn:gm-am-temp} twice, we arrive at the bound
\begin{equation} \label{eqn:gm-am-temp2}
\begin{aligned}
\trace \big[ \mtx{H}\mtx{W}^q \mtx{H} \mtx{Y}^{2r-q} \big]
+ \trace \big[ \mtx{H}\mtx{W}^{2r-q} \mtx{H} \mtx{Y}^{q} \big]
	&\leq \sum_{i,j=1}^d \big( \abs{\lambda_i}^q \abs{\smash{\mu_j}}^{2r-q}
	+ \abs{\lambda_i}^{2r-q} \abs{\smash{\mu_j}}^{q}\big)
	\cdot \abssq{ \vct{u}_i^\adj \mtx{H} \vct{v}_j } \\
	&\leq \sum_{i,j=1}^d \big( \lambda_i^{2r} + \mu_j^{2r} \big)
	\cdot \abssq{ \vct{u}_i^\adj \mtx{H} \vct{v}_j }.
\end{aligned}
\end{equation}
The second inequality is~\eqref{eqn:heinz}, with $\theta = q/(2r)$ and $\lambda = \lambda_i^{2r}$
and $\mu = \mu_j^{2r}$.

It remains to rewrite the right-hand side of~\eqref{eqn:gm-am-temp2} a more recognizable form.  To that end,
observe that
$$
\begin{aligned}
\trace \big[ \mtx{H}\mtx{W}^q \mtx{H} \mtx{Y}^{2r-q} \big]
&+ \trace \big[ \mtx{H}\mtx{W}^{2r-q} \mtx{H} \mtx{Y}^{q} \big] \\
	&\leq \sum_{i,j=1}^d \big( \lambda_i^{2r} + \mu_j^{2r} \big)
	\cdot \trace \big[ \mtx{H} \vct{u}_i \vct{u}_i^\adj \mtx{H} \vct{v}_j \vct{v}_j^\adj \big] \\
	&= \trace \big[ \mtx{H} \big( \sum_{i=1}^d \lambda_i^{2r} \vct{u}_i \vct{u}_i^\adj \big) \mtx{H} 
	\big( \sum_{j=1}^d \vct{v}_j \vct{v}_j^\adj \big) \big]
	+ \trace \big[ \mtx{H} \big( \sum_{i=1}^d \vct{u}_i \vct{u}_i^\adj \big)
	\mtx{H} \big( \sum_{j=1}^d \mu_j^{2r} \vct{v}_j \vct{v}_j^\adj \big) \big] \\
	&= \trace \big[ \mtx{H}^2 \cdot \mtx{W}^{2r} \big] + \trace \big[ \mtx{H}^2 \cdot \mtx{Y}^{2r} \big].
\end{aligned}
$$
In the first step, we return the squared magnitude to its representation as a trace.  Then
we use linearity to draw the sums back inside the trace.  Next, invoke~\eqref{eqn:monomial}
to identify the powers $\mtx{W}^{2r}$ and $\mtx{Y}^{0} = \Id_d$ and $\mtx{W}^0 = \Id_d$ and $\mtx{Y}^{2r}$.
Last, use the cyclicity of the trace to combine the factors of $\mtx{H}$.
The result~\eqref{eqn:trace-gm-am} follows from the linearity of the trace.
\end{proof}

\begin{remark}[The Power of Abstraction]
There is a cleaner, but more abstract, proof of the inequality~\eqref{eqn:trace-gm-am}.
Consider the left- and right-multiplication operators
$$
\mathsf{W} : \mtx{H} \mapsto \mtx{WH}
\quad\text{and}\quad
\mathsf{Y} : \mtx{H} \mapsto \mtx{HY}.
$$
Observe that powers of $\mathsf{W}$ and $\mathsf{Y}$
correspond to left- and right-multiplication by powers of $\mtx{W}$ and $\mtx{Y}$. 
Now, the operators $\mathsf{W}$ and $\mathsf{Y}$ commute,
so there is a basis (orthonormal with respect to the trace inner product) for $\Sym_d$
in which they are simultaneously diagonalizable.
Representing the operators in this basis, we can use~\eqref{eqn:heinz} to check that
$$
\mathsf{W}^q \mathsf{Y}^{2r-q}
+ \mathsf{W}^{2r-q} \mathsf{Y}^{q}
\psdle \mathsf{W}^{2r} + \mathsf{Y}^{2r}.
$$
Now, calculate that
$$
\begin{aligned}
\trace\big[ \mtx{HW}^q \mtx{HY}^{2r-q} \big] + \trace\big[ \mtx{HW}^{2r-q} \mtx{HY}^{q} \big]
	&= \trace\big[ \mtx{H} \cdot \big(\mathsf{W}^q \mathsf{Y}^{2r-q} + \mathsf{W}^{2r-q} \mathsf{Y}^q \big)(\mtx{H}) \big] \\\
	&\leq \trace\big[ \mtx{H} \cdot \big(\mathsf{W}^{2r} + \mathsf{Y}^{2r} \big)(\mtx{H}) \big] \\
	&= \trace\big[ \mtx{H}^2 \cdot \big( \mtx{W}^{2r} + \mtx{Y}^{2r} \big) \big].
\end{aligned}
$$
We omit the details.
\end{remark}

\subsection{The Hermitian Dilation}

Last, we introduce the \term{Hermitian dilation} $\coll{H}(\mtx{B})$ of a rectangular matrix $\mtx{B} \in \mathbb{M}^{d_1 \times d_2}$.  This is the Hermitian matrix
\begin{equation} \label{eqn:dilation}
\coll{H}(\mtx{B}) := \begin{bmatrix} \mtx{0} & \mtx{B} \\ \mtx{B}^\adj & \mtx{0} \end{bmatrix}
	\in \Sym_{d_1 + d_2}.
\end{equation}
Note that the map $\coll{H}$ is real linear.  By direct calculation,
\begin{equation} \label{eqn:dilation-square}
\coll{H}(\mtx{B})^2 = \begin{bmatrix} \mtx{BB}^\adj & \mtx{0} \\ \mtx{0} & \mtx{B}^\adj \mtx{B} \end{bmatrix}.
\end{equation}
We also have the spectral-norm identity
\begin{equation} \label{eqn:dilation-norm}
\norm{ \coll{H}(\mtx{B}) } = \norm{ \mtx{B} }.
\end{equation}
To verify~\eqref{eqn:dilation-norm}, calculate that
$$
\normsq{\coll{H}(\mtx{B})}
	= \norm{ \coll{H}(\mtx{B})^2 }
	= \norm{ \begin{bmatrix} \mtx{BB}^\adj & \mtx{0} \\ \mtx{0} & \mtx{B}^\adj \mtx{B} \end{bmatrix} }
	= \max\big\{ \norm{ \smash{\mtx{BB}^\adj} }, \ \norm{ \smash{\mtx{B}^\adj \mtx{B}} } \big\} 
	= \normsq{ \mtx{B} }.
$$
The first identity is~\eqref{eqn:norm-power}; the second is~\eqref{eqn:dilation-square}.
The norm of a block-diagonal Hermitian matrix is the maximum spectral norm of a block,
which follows from the Rayleigh principle~\eqref{eqn:rayleigh} with a bit of work.
Finally, invoke the property~\eqref{eqn:B*}.

\section{Probability Background} \label{sec:probability}

This section contains some background material from
the field of probability.  Good references include the
books~\cite{LT91:Probability-Banach,GS01:Probability-Random,Tao12:Topics-Random}.

\subsection{Expectation}

The symbol $\Expect$ denotes the expectation operator.  We will not
define expectation formally or spend any energy on technical details.
No issues arise if we assume, for example, that all random variables are bounded.

We use brackets to enclose the argument
of the expectation when it is important for clarity, and we instate
the convention that nonlinear functions bind before expectation.
For instance, $\Expect X^p := \Expect[ X^p ]$ and
$\Expect \max\nolimits_i X_i := \Expect[ \max\nolimits_i X_i ]$.

Sometimes, we add a subscript to indicate a partial expectation.  For example,
if $J$ is a random variable, $\Expect_{J}$ refers to the average over $J$,
with all other random variables fixed.  We only use this notation
when $J$ is independent from the other random variables, so there
are no complications.  In particular, we can compute iterated expectations:
$\Expect[ \Expect_J [ \cdot ] ] = \Expect[ \cdot ] $ whenever all the
expectations are finite.

\subsection{Random Matrices}

A \term{random matrix} is a matrix whose entries are complex random variables,
not necessarily independent.
We compute the expectation of a random matrix $\mtx{Z}$ componentwise:
\begin{equation*} \label{eqn:expect-matrix}
(\Expect[ \mtx{Z} ] )_{ij} = \Expect[ z_{ij} ]
\quad\text{for each pair $(i, j)$ of indices.}
\end{equation*}
As in the scalar case, if $\mtx{W}$ and $\mtx{Z}$ are independent,
$$
\Expect[ \mtx{WZ} ] = (\Expect \mtx{W} )(\Expect \mtx{Z}).
$$
Since the expectation is linear, it also commutes with all of the simple linear
operations we perform on matrices.

It suffices to take a na{\"i}ve view of independence, expectation, and so forth.
For the technically inclined, let $(\Omega, \coll{F}, \mathbb{P})$
be a probability space.  A $d_1 \times d_2$ random matrix $\mtx{Z}$
is simply a measurable function
$$
\mtx{Z} : \Omega \to \mathbb{M}^{d_1 \times d_2}.
$$
A family $\{ \mtx{Z}_i : i = 1, \dots, n \}$ of random matrices is independent
when
$$
\Prob{ \mtx{Z}_i \in E_i \text{ for $i = 1, \dots, n$} } = \prod_{i=1}^n \Prob{ \mtx{Z}_i \in E_i }.
$$
for any collection of Borel\footnote{Open sets in $\mathbb{M}^{d_1 \times d_2}$
are defined with respect to the metric topology induced by the spectral norm.}
subsets $E_i \subset \mathbb{M}^{d_1 \times d_2}$.

\subsection{Inequalities for Expectation}

We need several basic inequalities for expectation.  We set these
out for future reference.
Let $X, Y$ be (arbitrary) real random variables.  The Cauchy--Schwarz inequality states that
\begin{equation} \label{eqn:expect-cs}
\abs{ \Expect[ XY ] } \leq \big(\Expect X^2 \big)^{1/2} \cdot \big(\Expect Y^2 \big)^{1/2}.
\end{equation}
For $r \geq 1$, the triangle inequality states that
\begin{equation} \label{eqn:Lr-triangle}
\big( \Expect \abs{ X + Y }^r \big)^{1/r}
	\leq \big( \Expect \abs{X}^r \big)^{1/r} + \big( \Expect \abs{Y}^r \big)^{1/r}.
\end{equation}
Each of these inequalities is vacuous precisely when its right-hand side is infinite.

Jensen's inequality describes how expectation interacts with a convex or concave function;
cf.~\eqref{eqn:convexity}.
Let $X$ be a random variable taking values in a finite-dimensional linear space $V$,
and let $f : V \to \R$ be a function.  Then
\begin{equation} \label{eqn:jensen}
\begin{aligned}
f(\Expect X) &\leq \Expect f(X) \quad\text{when $f$ is convex, and} \\
\Expect f(X) &\leq f(\Expect X) \quad\text{when $f$ is concave.} \\
\end{aligned}
\end{equation}
The inequalities~\eqref{eqn:jensen} also hold when we replace $\Expect$ with a partial expectation.
Let us emphasize that these bounds do require that all of the expectations exist.

\subsection{Symmetrization}

Symmetrization is an important technique for studying the expectation
of a function of independent random variables.  The idea is to inject
auxiliary randomness into the function.  Then we condition on the
original random variables and average with respect to the extra
randomness.  When the auxiliary random variables are more pliable,
this approach can lead to significant simplifications.

A \term{Rademacher} random variable $\eps$ takes the two values $\pm 1$ with
equal probability.  The following result shows how we can use Rademacher random variables to
study a sum of independent random matrices.

\begin{fact}[Symmetrization] \label{fact:symmetrization}
Let $\mtx{S}_1, \dots, \mtx{S}_n \in \mathbb{M}^{d_1 \times d_2}$
be independent random matrices.
Let $\eps_1, \dots, \eps_n$ be independent Rademacher random variables that are
also independent from the random matrices.  For each $r \geq 1$,
$$
\frac{1}{2} \cdot \left( \Expect \norm{ \sum_{i=1}^n \eps_i \mtx{S}_i }^r \right)^{1/r}
	\leq \left( \Expect \norm{ \sum_{i=1}^n (\mtx{S}_i - \Expect \mtx{S}_i) }^r \right)^{1/r}
	\leq 2 \cdot \left( \Expect \norm{ \sum_{i=1}^n \eps_i \mtx{S}_i }^r \right)^{1/r}.
$$
This result holds whenever $\Expect \norm{\mtx{S}_i}^r < \infty$ for each index $i$.
\end{fact}

\begin{proof}
For notational simplicity, assume that $r = 1$.  We discuss the general case
at the end of the argument.

Let $\{\mtx{S}_i' : i = 1, \dots, n\}$ be an independent copy of the sequence $\{ \mtx{S}_i : i = 1, \dots, n \}$,
and let $\Expect'$ denote partial expectation with respect to the independent copy.  Then
$$
\begin{aligned}
\Expect \norm{ \sum_{i=1}^n ( \mtx{S}_i - \Expect \mtx{S}_i) }
	&= \Expect \norm{ \sum_{i=1}^n \big[ ( \mtx{S}_i - \Expect \mtx{S}_i)
	- \Expect' ( \mtx{S}'_i - \Expect \mtx{S}_i) \big] } \\
	&\leq \Expect \left[ \Expect' \norm{ \sum_{i=1}^n \big[ ( \mtx{S}_i - \Expect \mtx{S}_i)
	- ( \mtx{S}'_i - \Expect \mtx{S}_i) \big] } \right] \\
	&= \Expect \norm{ \sum_{i=1}^n (\mtx{S}_i - \mtx{S}_i') }.
\end{aligned}
$$
The first identity holds because $\Expect' \mtx{S}'_i = \Expect \mtx{S}_i$ by identical distribution.
Since the spectral norm is convex, we can apply Jensen's inequality~\eqref{eqn:jensen} conditionally
to draw out the partial expectation $\Expect'$.  Last, we combine the iterated expectation into
a single expectation.

Observe that $\mtx{S}_i - \mtx{S}'_i$ has the same distribution as its negation $\mtx{S}_i' - \mtx{S}_i$.
It follows that the independent sequence
$\{ \eps_i (\mtx{S}_i - \mtx{S}_i') : i = 1, \dots, n \}$
has the same distribution as $\{ \mtx{S}_i - \mtx{S}_i' : i = 1, \dots, n \}$.
Therefore, the expectation of any nonnegative function
takes the same value for both sequences.  In particular,
$$
\begin{aligned}
\Expect \norm{ \sum_{i=1}^n (\mtx{S}_i - \mtx{S}_i') }
	&= \Expect \norm{ \sum_{i=1}^n \eps_i (\mtx{S}_i - \mtx{S}_i') } \\
	&\leq \Expect \norm{ \sum_{i=1}^n \eps_i \mtx{S}_i } + \Expect \norm{ \sum_{i=1}^n (-\eps_i) \mtx{S}'_i } \\
	&= 2 \Expect \norm{ \sum_{i=1}^n \eps_i \mtx{S}_i }.
\end{aligned}
$$
The second step is the triangle inequality, and the last line follows from the identical distribution
of $\{ \eps_i \mtx{S}_i \}$ and $\{ - \eps_i \mtx{S}_i' \}$.  Combine the last two displays to obtain
the upper bound.

To obtain results for $r > 1$,
we pursue the same approach.  We require the additional observation that $\norm{ \cdot }^r$
is a convex function, and we also need to invoke the triangle inequality~\eqref{eqn:Lr-triangle}.
Finally, we remark that the lower bound follows from a similar procedure, so we omit
the demonstration.
\end{proof}

\section{The Expected Norm of a Matrix Rademacher Series} \label{sec:khintchine}

To prove Theorem~\ref{thm:main}, our overall strategy is to use symmetrization.
This approach allows us to reduce the study of an independent sum of random matrices
to the study of a sum of fixed matrices modulated by independent
Rademacher random variables.  This type of random matrix is
called a \term{matrix Rademacher series}.  In this section,
we establish a bound on the spectral norm of a matrix Rademacher series.
This is the key technical step in the proof of Theorem~\ref{thm:main}.

\begin{theorem}[Matrix Rademacher Series] \label{thm:matrix-rademacher}
Let $\mtx{H}_1, \dots, \mtx{H}_n$ be fixed Hermitian matrices with dimension $d$.
Let $\eps_1, \dots, \eps_n$ be independent Rademacher random variables.  Then
\begin{equation} \label{eqn:rad-result}
\left( \Expect \normsq{ \sum_{i=1}^n \eps_i \mtx{H}_i } \right)^{1/2}
	\leq \sqrt{1 + 2\lceil\log d\rceil} \cdot \norm{ \sum_{i=1}^n \mtx{H}_i^2 }^{1/2}.
\end{equation}
\end{theorem}

\noindent
The proof of Theorem~\ref{thm:matrix-rademacher} occupies the bulk
of this section, beginning with Section~\ref{sec:rad-pf}.
The argument is really just a fancy version of the familiar
calculation of the moments of a centered standard normal
random variable; see Section~\ref{sec:ibp} for details.

\subsection{Discussion}

Before we establish Theorem~\ref{thm:matrix-rademacher},
let us make a few comments.
First, it is helpful to interpret the result in the same language
we have used to state Theorem~\ref{thm:main}.  Introduce the
matrix Rademacher series
$$
\mtx{X} := \sum_{i=1}^n \eps_i \mtx{H}_i.
$$
Compute the matrix variance, defined in~\eqref{eqn:variance-param}:
$$
v(\mtx{X}) := \norm{ \Expect \mtx{X}^2 }
	= \norm{ \sum_{i,j=1}^n \Expect[ \eps_i \eps_j ] \cdot \mtx{H}_i \mtx{H}_j }
	= \norm{ \sum_{i=1}^n \mtx{H}_i^2 }.
$$
We may rewrite Theorem~\ref{thm:matrix-rademacher} as the statement that
$$
\left( \Expect \normsq{ \mtx{X} } \right)^{1/2}
	\leq \sqrt{(1 + 2 \lceil\log d\rceil) \cdot v(\mtx{X})}.
$$
In other words, Theorem~\ref{thm:matrix-rademacher} is a sharper
version of Theorem~\ref{thm:main} for the special case of a 
matrix Rademacher series.

Next, we have focused on bounding the second moment of
$\norm{\mtx{X}}$ because this is the
most natural form of the result.  Note that we also
control the first moment because of Jensen's inequality~\eqref{eqn:jensen}:
\begin{equation} \label{eqn:rad-first-moment}
\Expect \norm{ \sum_{i=1}^n \eps_i \mtx{H}_i }
	\leq \left( \Expect \normsq{ \sum_{i=1}^n \eps_i \mtx{H}_i } \right)^{1/2}
	\leq \sqrt{1 + 2\lceil\log d\rceil} \cdot \norm{ \sum_{i=1}^n \mtx{H}_i^2 }^{1/2}.
\end{equation}
A simple variant on the proof of Theorem~\ref{thm:matrix-rademacher}
provides bounds for higher moments.

Third, the dimensional factor on the right-hand side of~\eqref{eqn:rad-result}
is asymptotically sharp.  Indeed, let us write $K(d)$ for the minimum possible constant
in the inequality
$$
\left( \Expect \normsq{ \sum_{i=1}^n \eps_i \mtx{H}_i } \right)^{1/2}
	\leq K(d) \cdot \norm{ \sum_{i=1}^n \mtx{H}_i^2 }^{1/2}
	\quad\text{for $\mtx{H}_i \in \Sym_d$ and $n \in \mathbb{N}$.}
$$
The example in Section~\ref{sec:upper-var} shows that
$$
K(d) \geq \sqrt{2 \log d}.
$$
In other words,~\eqref{eqn:rad-result} cannot be improved without making further assumptions.

Theorem~\ref{thm:matrix-rademacher} is a variant on the noncommutative
Khintchine inequality, first established by Lust-Piquard~\cite{LP86:Inegalites-Khintchine}
and later improved by Pisier~\cite{Pis98:Noncommutative-Vector} and by Buchholz~\cite{Buc01:Operator-Khintchine}.
The noncommutative Khintchine inequality gives bounds for the Schatten norm of a matrix
Rademacher series, rather than for the spectral norm.
Rudelson~\cite{Rud99:Random-Vectors} pointed out that
the noncommutative Khintchine inequality also implies bounds
for the spectral norm of a matrix Rademacher series.
In our presentation, we choose to control the spectral norm directly.

\subsection{The Spectral Norm and the Trace Moments}
\label{sec:rad-pf}

To begin the proof of Theorem~\ref{thm:matrix-rademacher},
we introduce the random Hermitian matrix
\begin{equation} \label{eqn:rad-series}
\mtx{X} := \sum_{i=1}^n \eps_i \mtx{H}_i
\end{equation}
Our goal is to bound the expected spectral norm of $\mtx{X}$.
We may proceed by estimating the expected trace of a power
of the random matrix, which is known as a \term{trace moment}.
Fix a nonnegative integer $p$.  Observe that
\begin{equation} \label{eqn:norm-trace-moments}
\left( \Expect \normsq{\mtx{X}} \right)^{1/2}
	\leq \big( \Expect \norm{ \mtx{X} }^{2p} \big)^{1/(2p)}
	= \big( \Expect \norm{ \smash{\mtx{X}^{2p}} } \big)^{1/(2p)}
	\leq \left( \Expect \trace \mtx{X}^{2p} \right)^{1/(2p)}.
\end{equation}
The first identity is Jensen's inequality~\eqref{eqn:jensen},
applied to the concave function $t \mapsto t^{1/p}$.
The second relation is~\eqref{eqn:norm-power}.
The final inequality is the bound~\eqref{eqn:norm-trace-bd}
on the norm of the positive-semidefinite matrix $\mtx{X}^{2p}$ by its trace.

\begin{remark}[Higher Moments]
It should be clear that we can also bound expected powers of the
spectral norm using the same technique.  For simplicity, we omit
this development.
\end{remark}

\subsection{Summation by Parts}

To study the trace moments of the random matrix $\mtx{X}$,
we rely on a discrete analog of integration by parts.
This approach is clearer if we introduce some more notation.
For each index $i$, define the random matrices
$$
\mtx{X}_{+i} := \mtx{H}_i + \sum_{j\neq i} \eps_j \mtx{H}_j
\quad\text{and}\quad
\mtx{X}_{-i} := - \mtx{H}_i + \sum_{j\neq i} \eps_j \mtx{H}_j
$$
In other words, the distribution of $\mtx{X}_{\eps_i i}$ is the conditional distribution
of the random matrix $\mtx{X}$ given the value $\eps_i$ of the $i$th Rademacher variable.
This interpretation depends on the assumption that the Rademacher variables are independent.

Beginning with the trace moment, observe that
\begin{equation} \label{eqn:sum-parts}
\begin{aligned}
\Expect\trace \mtx{X}^{2p}
	&= \Expect \trace \big[ \mtx{X} \cdot \mtx{X}^{2p-1} \big] \\
	&= \sum_{i=1}^n \Expect \big[ \Expect_{\eps_i} \trace \big[ \eps_i \mtx{H}_i \cdot \mtx{X}^{2p-1} \big] \big] \\
	&= \frac{1}{2} \sum_{i=1}^n \Expect \trace\Big[ \mtx{H}_i \cdot \Big(\mtx{X}_{+i}^{2p-1} - \mtx{X}_{-i}^{2p-1} \Big) \Big]
\end{aligned}	
\end{equation}
In the second step, we simply write out the definition~\eqref{eqn:rad-series}
of the random matrix $\mtx{X}$ and use the linearity of the trace to draw out the sum.
Then we write the expectation as an iterated expectation.  To reach the next line,
write out the partial expectation using the notation $\mtx{X}_{\pm i}$
and the linearity of the trace.

\subsection{A Difference of Powers}

Next, let us apply an algebraic identity to reduce the difference of powers
in~\eqref{eqn:sum-parts}.  For matrices $\mtx{W}, \mtx{Y} \in \Sym_d$,
it holds that
\begin{equation} \label{eqn:diff-powers}
\mtx{W}^{2p-1} - \mtx{Y}^{2p-1} = \sum_{q=0}^{2p-2} \mtx{W}^{q} (\mtx{W} - \mtx{Y}) \mtx{Y}^{2p-2-q}.
\end{equation}
To check this expression, just expand the matrix products and notice that the sum telescopes.

Introducing the relation~\eqref{eqn:diff-powers} with $\mtx{W} = \mtx{X}_{+i}$ and $\mtx{Y} = \mtx{X}_{-i}$
into the formula~\eqref{eqn:sum-parts}, we find that
\begin{equation} \label{eqn:alt-prod}
\begin{aligned}
\Expect\trace \mtx{X}^{2p}
	&= \frac{1}{2} \sum_{i=1}^n \Expect \trace \left[ \mtx{H}_i
	\cdot \sum_{q=0}^{2p-2} \mtx{X}_{+i}^{q} \big(\mtx{X}_{+i} - \mtx{X}_{-i} \big) \mtx{X}_{-i}^{2p-2-q} \right] \\
	&= \sum_{i=1}^n \sum_{q=0}^{2p-2} \Expect \trace
	\left[ \mtx{H}_i \mtx{X}_{+i}^{q} \mtx{H}_i \mtx{X}_{-i}^{2p-2-q} \right].
\end{aligned}
\end{equation}
Linearity of the trace allows us to draw out the sum over $q$,
and we have used the observation that $\mtx{X}_{+i} - \mtx{X}_{-i} = 2 \mtx{H}_i$.

\subsection{A Bound for the Trace Moments}

We are now in a position to obtain a bound for the trace moments of $\mtx{X}$.
Beginning with~\eqref{eqn:alt-prod}, we compute that
\begin{equation} \label{eqn:post-gm-am}
\begin{aligned}
\Expect \trace \mtx{X}^{2p}
	&= \sum_{i=1}^n \sum_{q=0}^{2p-2} \Expect \trace
	\left[ \mtx{H}_i \mtx{X}_{+i}^{q} \mtx{H}_i \mtx{X}_{-i}^{2p-2-q} \right] \\
	&\leq \sum_{i=1}^n \frac{2p-1}{2} \Expect \trace \left[ \mtx{H}_i^2 \cdot
	\Big( \mtx{X}_{+i}^{2p-2} + \mtx{X}_{-i}^{2p-2} \Big) \right] \\
	&= (2p-1) \cdot \sum_{i=1}^n  \Expect \trace \left[ \mtx{H}_i^2 \cdot \big( \Expect_{\eps_i} \mtx{X}^{2p-2} \big) \right] \\
	&= (2p-1) \cdot \Expect \trace \left[ \left( \sum_{i=1}^n \mtx{H}_i^2 \right) \cdot \mtx{X}^{2p-2} \right] \\
	&\leq (2p-1) \cdot \norm{ \sum_{i=1}^n \mtx{H}_i^2 } \cdot \Expect \trace \mtx{X}^{2p-2}.
\end{aligned}
\end{equation}
The bound in the second line is the trace GM--AM inequality, Fact~\ref{fact:trace-gm-am},
with $r = p - 1$ and $\mtx{W} = \mtx{X}_{+i}$ and $\mtx{Y} = \mtx{X}_{-i}$.
To reach the third line, observe that the parenthesis in the second line
is twice the partial expectation of $\mtx{X}^{2p-2}$ with respect to $\eps_i$.
Afterward, we use linearity of the expectation and the trace to draw in the sum over $i$,
and then we combine the expectations.
Last, invoke the trace inequality from Fact~\ref{fact:trace-dual}.

\subsection{Iteration and the Spectral Norm Bound}

The expression~\eqref{eqn:post-gm-am} shows that the trace moment
is controlled by a trace moment with a smaller power:
\begin{equation*} \label{eqn:iteration}
\Expect \trace \mtx{X}^{2p}
	\leq (2p-1) \cdot \norm{ \sum_{i=1}^n \mtx{H}_i^2 } \cdot \Expect \trace\mtx{X}^{2p-2}.
\end{equation*}
Iterating this bound $p$ times, we arrive at the result
\begin{equation} \label{eqn:iterated}
\begin{aligned}
\Expect \trace \mtx{X}^{2p}
	&\leq (2p-1)!! \cdot \norm{ \sum_{i=1}^n \mtx{H}_i^2 }^p \cdot \trace \mtx{X}^0 \\
	&= d \cdot (2p-1)!! \cdot \norm{ \sum_{i=1}^n \mtx{H}_i^2 }^p.
\end{aligned}
\end{equation}
The double factorial is defined as $(2p-1)!! := (2p-1) (2p-3)(2p-5) \cdots (5)(3)(1)$.

The expression~\eqref{eqn:norm-trace-moments}
shows that we can control the expected spectral
norm of $\mtx{X}$ by means of a trace moment.
Therefore, for any nonnegative integer $p$, it holds that
\begin{equation} \label{eqn:rad-norm-bd}
\Expect \norm{\mtx{X}}
	\leq \big( \Expect \trace \mtx{X}^{2p} \big)^{1/(2p)}
	\leq \big( d \cdot (2p-1)!! \big)^{1/(2p)} \cdot \norm{ \sum_{i=1}^n \mtx{H}_i^2 }^{1/2}.
\end{equation}
The second inequality is simply our bound~\eqref{eqn:iterated}.
All that remains is to choose the value of $p$ to minimize
the factor on the right-hand side.

\subsection{Calculating the Constant}

Finally, let us develop an accurate bound for the leading factor on the right-hand side of~\eqref{eqn:rad-norm-bd}.
We claim that
\begin{equation} \label{eqn:doublefactorial-claim}
(2p-1)!! \leq \left( \frac{2p+1}{\econst} \right)^{p}.
\end{equation}
Given this estimate, select $p = \lceil \log d \rceil$ to reach
\begin{equation} \label{eqn:rad-const-bd}
\big( d \cdot (2p-1)!! \big)^{1/(2p)}
	\leq d^{1/(2p)} \sqrt{\frac{2p+1}{\econst}}
	\leq \sqrt{2p + 1}
	= \sqrt{1 + 2 \lceil \log d \rceil}.
\end{equation}
Introduce the inequality~\eqref{eqn:rad-const-bd} into~\eqref{eqn:rad-norm-bd}
to complete the proof of Theorem~\ref{thm:matrix-rademacher}.

To check that~\eqref{eqn:doublefactorial-claim} is valid,
we use some tools from integral calculus:
$$
\begin{aligned}
\log\big( (2p-1)!! \big) &= \sum_{i=1}^{p-1} \log(2i + 1) \\
	&= \left[ \frac{1}{2} \log(2 \cdot 0 + 1) + \sum_{i=1}^{p-1} \log(2i + 1) + \frac{1}{2} \log(2p + 1) \right]
	- \frac{1}{2} \log(2p+1) \\
	&\leq \int_0^p \log(2x + 1) \idiff{x} - \frac{1}{2} \log(2p+1) \\
	&= p \log(2p+1) - p. 
\end{aligned}
$$
The bracket in the second line is the trapezoid rule approximation of the integral in the third line.
Since the integrand is concave, the trapezoid rule underestimates the integral.
Exponentiating this formula, we arrive at~\eqref{eqn:doublefactorial-claim}.

\subsection{Context} \label{sec:ibp}

The proof of Theorem~\ref{thm:matrix-rademacher}
is really just a discrete, matrix version of the familiar calculation of the
$(2p)$th moment of a centered normal random variable.
Let us elaborate.
Recall the Gaussian integration by parts formula:
\begin{equation} \label{eqn:gauss-ibp}
\Expect[ \gamma \cdot f(\gamma) ] = \sigma^2 \cdot \Expect[ f'(\gamma) ]
\end{equation}
where $\gamma \sim \normal(0, \sigma^2)$ and $f : \R \to \R$ is any function
for which the integrals are finite.  This result follows when we write
the expectations as integrals with respect to the normal density
tand invoke the usual integration by parts rule.
Now, suppose that we wish to compute
the $(2p)$th moment of $\gamma$.  We have
\begin{equation} \label{eqn:gauss-2p}
\Expect \gamma^{2p} = \Expect\big[ \gamma \cdot \gamma^{2p-1} \big]
	= (2p-1) \cdot \sigma^2 \cdot \Expect \gamma^{2p-2}.
\end{equation}
The second identity is just~\eqref{eqn:gauss-ibp} with the choice $f(t) = t^{2p-1}$.
Iterating~\eqref{eqn:gauss-2p}, we discover that
$$
\Expect \gamma^{2p} = (2p-1)!! \cdot \sigma^{2p}.
$$
In Theorem~\ref{thm:matrix-rademacher},
the matrix variance parameter $v(\mtx{X})$ plays
the role of the scalar variance $\sigma^2$.

In fact, the link with Gaussian integration by parts is even stronger.
Consider a matrix Gaussian series
$$
\mtx{Y} := \sum_{i=1}^n \gamma_i \mtx{H}_i
$$
where $\{ \gamma_i \}$ is an independent family of standard normal variables.
If we replace the discrete integration by parts in the proof of Theorem~\ref{thm:matrix-rademacher}
with Gaussian integration by parts, the argument leads to the bound
$$
\left( \Expect \normsq{ \sum_{i=1}^n \gamma_i \mtx{H}_i } \right)^{1/2}
	\leq \sqrt{1+ 2\lceil \log d \rceil} \cdot \norm{ \sum_{i=1}^n \mtx{H}_i^2 }^{1/2}.
$$
This approach requires matrix calculus, but it is slightly simpler
than the argument for matrix Rademacher series in other respects.
See~\cite[Thm.~8.1]{Tro15:Second-Order-Matrix} for a proof of
the noncommutative Khintchine inequality for Gaussian series
along these lines.

\section{Upper Bounds for the Expected Norm} \label{sec:upper}

We are now prepared to establish the upper bound for an arbitrary sum of independent random matrices.
The argument is based on the specialized result, Theorem~\ref{thm:matrix-rademacher},
for matrix Rademacher series.  It proceeds by steps through more and more general classes of
random matrices: first positive semidefinite, then Hermitian, and finally rectangular.
Here is what we will show.

\begin{theorem}[Expected Norm: Upper Bounds] \label{thm:upper}
Define the dimensional constant $C(d) := 4(1 + 2\lceil \log d \rceil)$.
The expected spectral norm of a sum of independent random matrices satisfies the following upper bounds.

\begin{enumerate}
\item	\textbf{The Positive-Semidefinite Case.}
Consider an independent family $\{ \mtx{T}_1, \dots, \mtx{T}_n \}$ of random $d \times d$ positive-semidefinite
matrices, and define the sum
$$
\mtx{W} := \sum_{i=1}^n \mtx{T}_i.
$$
Then \begin{equation} \label{eqn:psd-case}
\Expect \norm{ \mtx{W} }
	\leq \left[ \norm{ \Expect \mtx{W} }^{1/2} +
	\sqrt{C(d)} \cdot \big( \Expect \max\nolimits_i \norm{ \mtx{T}_i } \big)^{1/2} \right]^2.
\end{equation}

\item	\textbf{The Centered Hermitian Case.}
Consider an independent family $\{ \mtx{Y}_1, \dots, \mtx{Y}_n \}$ of random $d \times d$ Hermitian
matrices with $\Expect \mtx{Y}_i = \mtx{0}$ for each index $i$, and define the sum
$$
\mtx{X} := \sum_{i=1}^n \mtx{Y}_i.
$$
Then \begin{equation} \label{eqn:herm-case}
\left( \Expect \normsq{ \mtx{X} } \right)^{1/2}
	\leq \sqrt{C(d)} \cdot \norm{ \Expect \mtx{X}^2 }^{1/2}
	+ C(d) \cdot \left( \Expect \max\nolimits_i \norm{ \mtx{Y}_i }^2 \right)^{1/2}.
\end{equation}

\item	\textbf{The Centered Rectangular Case.}
Consider an independent family $\{ \mtx{S}_1, \dots, \mtx{S}_n \}$ of random $d_1 \times d_2$
matrices with $\Expect \mtx{S}_i = \mtx{0}$ for each index $i$, and define the sum
$$
\mtx{Z} := \sum_{i=1}^n \mtx{S}_i.
$$
Then \begin{equation} \label{eqn:rect-case}
\Expect \norm{ \mtx{Z} }
	\leq \sqrt{C(d)} \cdot \max\left\{ \norm{ \Expect\big[ \mtx{ZZ}^\adj \big] }^{1/2}, \
		\norm{ \Expect\big[ \mtx{Z}^\adj \mtx{Z} \big] }^{1/2} \right\}
		+ C(d) \cdot \left( \Expect \max\nolimits_i \norm{ \mtx{S}_i }^2 \right)^{1/2} 
\end{equation}
where $d := d_1 + d_2$.
\end{enumerate}
\end{theorem}

\noindent
The proof of Theorem~\ref{thm:upper} takes up the rest of this section.
The presentation includes notes about the provenance of various parts of the argument.

The upper bound in Theorem~\ref{thm:main} follows instantly from Case (3) of Theorem~\ref{thm:upper}.
We just introduce the notation $v(\mtx{Z})$ for the variance parameter,
and we calculate that
$$
\Expect\big[ \mtx{ZZ}^\adj \big]
	= \sum_{i,j=1}^n \Expect\big[ \mtx{S}_i \mtx{S}_j^\adj \big]
	= \sum_{i=1}^n \Expect\big[ \mtx{S}_i \mtx{S}_i^\adj \big].
$$
The first expression follows immediately from the definition of $\mtx{Z}$
and the linearity of the expectation; the second identity holds because
the random matrices $\mtx{S}_i$ are independent and have mean zero.
The formula for $\Expect\big[ \mtx{Z}^\adj \mtx{Z} \big]$
is valid for precisely the same reasons.

\subsection{Proof of the Positive-Semidefinite Case}

Recall that $\mtx{W}$ is a random $d \times d$ positive-semidefinite matrix of the form
$$
\mtx{W} := \sum_{i=1}^n \mtx{T}_i
\quad\text{where the $\mtx{T}_i$ are positive semidefinite.}
$$
Let us introduce notation for the quantity of interest:
$$
E := \Expect \norm{ \mtx{W} } = \Expect \norm{ \sum_{i=1}^n \mtx{T}_i }
$$
By the triangle inequality for the spectral norm,
$$
\begin{aligned}
E \leq \norm{ \sum_{i=1}^n \Expect \mtx{T}_i } + \Expect \norm{ \sum_{i=1}^n (\mtx{T}_i - \Expect \mtx{T}_i) }
	\leq \norm{ \sum_{i=1}^n \Expect \mtx{T}_i } + 2 \Expect \norm{ \sum_{i=1}^n \eps_i \mtx{T}_i }.
\end{aligned}
$$
The second inequality follows from symmetrization, Fact~\ref{fact:symmetrization}.
In this expression, $\{\eps_i\}$ is an independent family of Rademacher random variables,
independent from $\{ \mtx{T}_i \}$.
Conditioning on the choice of the random matrices $\mtx{T}_i$,
we apply Theorem~\ref{thm:matrix-rademacher}
via the bound~\eqref{eqn:rad-first-moment}:
$$
\begin{aligned}
\Expect \norm{ \sum_{i=1}^n \eps_i \mtx{T}_i }
	= \Expect\left[ \Expect_{\vct{\eps}} \norm{ \sum_{i=1}^n \eps_i \mtx{T}_i } \right]
	\leq \sqrt{1 + 2 \lceil\log d\rceil} \cdot \Expect\left[ \norm{ \sum_{i=1}^n \mtx{T}_i^2 }^{1/2} \right].
\end{aligned}
$$
The operator $\Expect_{\vct{\eps}}$ averages over the choice of the Rademacher random variables,
with the matrices $\mtx{T}_i$ fixed.
Now, since the matrices $\mtx{T}_i$ are positive-semidefinite,
$$
\begin{aligned}
\Expect\left[ \norm{ \sum_{i=1}^n \mtx{T}_i^2 }^{1/2} \right]
	&\leq \Expect \left[ \big(\max\nolimits_i \norm{ \mtx{T}_i } \big)^{1/2} \cdot \norm{ \sum_{i=1}^n \mtx{T}_i }^{1/2} \right] \\
	&\leq \big( \Expect \max\nolimits_i \norm{\mtx{T}_i} \big)^{1/2} \cdot
		\Big( \Expect \norm{ \sum_{i=1}^n \mtx{T}_i } \Big)^{1/2} \\
	&= \big( \Expect \max\nolimits_i \norm{\mtx{T}_i} \big)^{1/2} \cdot E^{1/2}.
\end{aligned}
$$
The first inequality is Fact~\ref{fact:sum-squares},
and the second is the Cauchy--Schwarz inequality~\eqref{eqn:expect-cs} for expectation.
In the last step, we identified a copy of the quantity $E$.

Combine the last three displays to see that
\begin{equation} \label{eqn:E-bd}
E \leq \norm{ \sum_{i=1}^n \Expect \mtx{T}_i }
	+ \sqrt{4(1 + 2\lceil\log d\rceil)} \cdot \big( \Expect \max\nolimits_i \norm{\mtx{T}_i} \big)^{1/2} \cdot E^{1/2}.
\end{equation}
For any $\alpha, \beta \geq 0$, the quadratic inequality $t^2 \leq \alpha + \beta t$ implies that
$$
t \leq \frac{1}{2} \left[ \beta + \sqrt{ \beta^2 + 4 \alpha } \right]
	\leq \frac{1}{2} \left[ \beta + \beta + 2 \sqrt{\alpha} \right]
	= \sqrt{\alpha} + \beta
$$
because the square root is subadditive.  Applying this fact to the quadratic relation~\eqref{eqn:E-bd}
for $E^{1/2}$, we obtain
$$
E^{1/2} \leq \norm{ \sum_{i=1}^n \Expect \mtx{T}_i }^{1/2}
	+ \sqrt{4(1+ 2\lceil \log d\rceil)} \cdot \big( \Expect \max\nolimits_i \norm{\mtx{T}_i} \big)^{1/2}.
$$
Square both sides to reach the conclusion~\eqref{eqn:psd-case}.

This argument is adapted from Rudelson's paper~\cite{Rud99:Random-Vectors},
which develops a version of this result for the case where the matrices
$\mtx{T}_i$ have rank one; see also~\cite{RV07:Sampling-Large}.
The paper~\cite{Tro08:Conditioning-Random} contains the first estimates
for the constants.  Magen \& Zouzias~\cite{MZ11:Low-Rank-Matrix-Valued}
observed that similar considerations apply when the matrices $\mtx{T}_i$ have higher
rank.  The complete result~\eqref{eqn:psd-case} first appeared
in~\cite[App.]{CGT12:Masked-Sample}.  The constants in this paper
are marginally better.
Related bounds for Schatten norms appear in~\cite[Sec.~7]{MJCFT14:Matrix-Concentration}
and in~\cite{JZ13:Noncommutative-Bennett}.

The results described in the last paragraph are all matrix versions of the classical inequalities
due to Rosenthal~\cite[Lem.~1]{Ros70:Subspaces-Lp}.  These bounds can be interpreted as polynomial 
moment versions of the Chernoff inequality.

\subsection{Proof of the Hermitian Case}

The result~\eqref{eqn:herm-case} for Hermitian matrices is a corollary
of Theorem~\ref{thm:matrix-rademacher} and the positive-semidefinite result~\eqref{eqn:psd-case}.
Recall that $\mtx{X}$ is a $d \times d$ random Hermitian matrix of the form
$$
\mtx{X} := \sum_{i=1}^n \mtx{Y}_i
\quad\text{where $\Expect \mtx{Y}_i = \mtx{0}$.}
$$
We may calculate that
$$
\begin{aligned}
\left( \Expect \normsq{ \mtx{X} } \right)^{1/2}
	&= \left( \Expect \normsq{ \sum_{i=1}^n \mtx{Y}_i } \right)^{1/2} \\
	&\leq 2 \left( \Expect \left[ \Expect_{\vct{\eps}} \normsq{ \sum_{i=1}^n \eps_i \mtx{Y}_i } \right] \right)^{1/2} \\
	&\leq \sqrt{4(1 + 2\lceil\log d \rceil)} \cdot \Big( \Expect \norm{ \sum_{i=1}^n \mtx{Y}_i^2 } \Big)^{1/2}.
\end{aligned}
$$
The first inequality follows from the symmetrization procedure, Fact~\ref{fact:symmetrization}.
The second inequality applies Theorem~\ref{thm:matrix-rademacher}, conditional
on the choice of $\mtx{Y}_i$.  
The remaining expectation contains a sum of independent positive-semidefinite matrices.  
Therefore, we may invoke~\eqref{eqn:psd-case} with $\mtx{T}_i = \mtx{Y}_i^2$.  We obtain
$$
\Expect \norm{ \sum_{i=1}^n \mtx{Y}_i^2 }
	\leq \left[ \norm{ \sum_{i=1}^n \Expect \mtx{Y}_i^2 }^{1/2}
	+ \sqrt{4(1 + 2\lceil \log d \rceil)} \cdot \big( \Expect \max\nolimits_i \norm{ \smash{\mtx{Y}_i^2} } \big)^{1/2} \right]^2.
$$
Combine the last two displays to reach
$$
\left( \Expect \normsq{ \mtx{X} } \right)^{1/2}
	\leq \sqrt{4(1 + 2\lceil\log d \rceil)} \cdot \left[
	\norm{ \sum_{i=1}^n \Expect \mtx{Y}_i^2 }^{1/2}
	+ \sqrt{4(1 + 2\lceil \log d \rceil)} \cdot
	\big( \Expect \max\nolimits_i \normsq{\mtx{Y}_i} \big)^{1/2} \right].
$$
Rewrite this expression to reach~\eqref{eqn:herm-case}.

A version of the result~\eqref{eqn:herm-case} first appeared in~\cite[App.]{CGT12:Masked-Sample};
the constants here are marginally better.  Related results for the
Schatten norm appear in the
papers~\cite{JX03:Noncommutative-Burkholder-I,JX08:Noncommutative-Burkholder-II,MJCFT14:Matrix-Concentration,JZ13:Noncommutative-Bennett}.
These bounds are matrix extensions of the scalar inequalities
due to Rosenthal~\cite[Thm.~3]{Ros70:Subspaces-Lp} and
to Ros{\'e}n~\cite[Thm.~1]{Ros70:Bounds-Central}; see also Nagaev--Pinelis~\cite[Thm.~2]{NP77:Some-Inequalities}.
They can be interpreted as the polynomial moment inequalities
that sharpen the Bernstein inequality.

\subsection{Proof of the Rectangular Case}
\label{sec:rect-case-upper}

Finally, we establish the rectangular result~\eqref{eqn:rect-case}.
Recall that $\mtx{Z}$ is a $d_1 \times d_2$ random rectangular matrix of the form
$$
\mtx{Z} := \sum_{i=1}^n \mtx{S}_i
\quad\text{where $\Expect \mtx{S}_i = \mtx{0}$.}
$$
Set $d := d_1 + d_2$, and form a random $d \times d$ Hermitian
matrix $\mtx{X}$ by dilating $\mtx{Z}$:
$$
\mtx{X} := \coll{H}(\mtx{Z}) = \sum_{i=1}^n \coll{H}(\mtx{S}_i).
$$
The Hermitian dilation $\coll{H}$ is defined in~\eqref{eqn:dilation};
the second relation holds because the dilation is a real-linear map.

Evidently, the random matrix $\mtx{X}$ is a sum of independent, centered, random
Hermitian matrices $\coll{H}(\mtx{S}_i)$.  Therefore, we may apply~\eqref{eqn:herm-case}
to $\mtx{X}$ to see that
\begin{equation} \label{eqn:rect-dilated}
\left( \Expect \normsq{\coll{H}(\mtx{Z})} \right)^{1/2}
	\leq \sqrt{4(1+2\lceil\log d\rceil)} \cdot \norm{ \Expect \big[  \coll{H}(\mtx{Z})^2 \big] }^{1/2}
	+ 4(1 + 2\lceil \log d \rceil) \cdot \big( \Expect \max\nolimits_i \normsq{ \coll{H}(\mtx{S}_i) } \big)^{1/2}.
\end{equation}
Since the dilation preserves norms~\eqref{eqn:dilation-norm}, the left-hand side
of~\eqref{eqn:rect-dilated} is exactly what we want:
$$
\left( \Expect \normsq{\coll{H}(\mtx{Z})} \right)^{1/2}
	= \left( \Expect \normsq{\mtx{Z}} \right)^{1/2}.
$$ 
To simplify the first term on the right-hand side of~\eqref{eqn:rect-dilated},
invoke the formula~\eqref{eqn:dilation-square} for the square of the dilation:
\begin{equation} \label{eqn:dilation-Z-square}
\norm{ \Expect\big[ \coll{H}(\mtx{Z})^2 \big] }
	= \norm{ \begin{bmatrix} \Expect\big[ \mtx{ZZ}^\adj \big] & \mtx{0} \\
	\mtx{0} & \Expect\big[ \mtx{Z}^\adj \mtx{Z} \big] \end{bmatrix} }
	= \max\left\{ \norm{ \Expect \big[ \mtx{ZZ}^\adj \big] }, \
	\norm{ \Expect \big[ \mtx{Z}^\adj \mtx{Z} \big] } \right\}.
\end{equation}
The second identity relies on the fact that the norm of a block-diagonal matrix
is the maximum norm of a diagonal block.
To simplify the second term on the right-hand side of~\eqref{eqn:rect-dilated},
we use~\eqref{eqn:dilation-norm} again:
$$
\norm{ \coll{H}(\mtx{S}_i) } = \norm{ \mtx{S}_i }.
$$
Introduce the last three displays into~\eqref{eqn:rect-dilated} to arrive at the result~\eqref{eqn:rect-case}.

The result~\eqref{eqn:rect-case} first appeared in the monograph~\cite[Eqn.~(6.16)]{Tro15:Introduction-Matrix}
with (possibly) incorrect constants.  The current paper contains the first complete presentation of the bound.

\section{Lower Bounds for the Expected Norm} \label{sec:lower}

Finally, let us demonstrate that each of the upper bounds in Theorem~\ref{thm:upper} is sharp
up to the dimensional constant $C(d)$.  The following result gives matching lower bounds
in each of the three cases.

\begin{theorem}[Expected Norm: Lower Bounds] \label{thm:lower}
The expected spectral norm of a sum of independent random matrices satisfies the following lower bounds.

\begin{enumerate}
\item	\textbf{The Positive-Semidefinite Case.}
Consider an independent family $\{ \mtx{T}_1, \dots, \mtx{T}_n \}$ of random $d \times d$ positive-semidefinite
matrices, and define the sum
$$
\mtx{W} := \sum_{i=1}^n \mtx{T}_i.
$$
Then \begin{equation} \label{eqn:psd-case-lower}
\Expect \norm{ \mtx{W} }
	\geq \frac{1}{4} \left[ \norm{ \Expect \mtx{W} }^{1/2} + \big(\Expect \max\nolimits_i \norm{ \mtx{T}_i } \big)^{1/2} \right]^2.
\end{equation}

\item	\textbf{The Centered Hermitian Case.}
Consider an independent family $\{ \mtx{Y}_1, \dots, \mtx{Y}_n \}$ of random $d \times d$ Hermitian
matrices with $\Expect \mtx{Y}_i = \mtx{0}$ for each index $i$, and define the sum
$$
\mtx{X} := \sum_{i=1}^n \mtx{Y}_i.
$$

Then \begin{equation} \label{eqn:herm-case-lower}
\left( \Expect \normsq{ \mtx{X} } \right)^{1/2}
	\geq \frac{1}{2} \norm{ \Expect \mtx{X}^2 }^{1/2}
	+ \frac{1}{4} \left( \Expect \max\nolimits_i \norm{ \mtx{Y}_i }^2 \right)^{1/2}.
\end{equation}

\item	\textbf{The Centered Rectangular Case.}
Consider an independent family $\{ \mtx{S}_1, \dots, \mtx{S}_n \}$ of random $d_1 \times d_2$
matrices with $\Expect \mtx{S}_i = \mtx{0}$ for each index $i$, and define the sum
$$
\mtx{Z} := \sum_{i=1}^n \mtx{S}_i.
$$
Then \begin{equation} \label{eqn:rect-case-lower}
\Expect \norm{ \mtx{Z} }
	\geq \frac{1}{2} \max\left\{ \norm{ \Expect\big[ \mtx{ZZ}^\adj \big] }^{1/2}, \
		\norm{ \Expect\big[ \mtx{Z}^\adj \mtx{Z} \big] } ^{1/2}\right\}
		+ \frac{1}{4} \left( \Expect \max\nolimits_i \norm{ \mtx{S}_i }^2 \right)^{1/2}.
\end{equation}
\end{enumerate}
\end{theorem}

\noindent
The rest of the section describes the proof of Theorem~\ref{thm:lower}.

The lower bound in Theorem~\ref{thm:main} is an immediate consequence of Case (3) of Theorem~\ref{thm:lower}.
We simply introduce the notation $v(\mtx{Z})$ for the variance parameter.

\subsection{The Positive-Semidefinite Case}

The lower bound~\eqref{eqn:psd-case-lower} in the positive-semidefinite case
is relatively easy.  Recall that
$$
\mtx{W} := \sum_{i=1}^n \mtx{T}_i
\quad\text{where the $\mtx{T}_i$ are positive semidefinite.}
$$
First, by Jensen's inequality~\eqref{eqn:jensen} and the convexity of the spectral norm,
\begin{equation} \label{eqn:psd-lower-mean}
\Expect \norm{ \mtx{W} } \geq \norm{ \Expect \mtx{W} }.
\end{equation}
Second, let $I$ be the minimum value of the index $i$ where $\max_i \norm{\mtx{T}_i}$ is achieved;
note that $I$ is a random variable.  Since the summands $\mtx{T}_i$ are positive semidefinite,
it is easy to see that
$$
\mtx{T}_I \psdle \sum_{i=1}^n \mtx{T}_i.
$$
Therefore, by the norm identity~\eqref{eqn:norm-psd} for a positive-semidefinite matrix
and the monotonicity of the maximum eigenvalue, Fact~\ref{fact:weyl}, we have
$$
\max\nolimits_i \norm{\mtx{T}_i}
	= \norm{\mtx{T}_I}
	= \lambda_{\max}(\mtx{T}_I)
	\leq \lambda_{\max}\left( \sum_{i=1}^n \mtx{T}_i \right)
	= \norm{ \sum_{i=1}^n \mtx{T}_i }
	= \norm{ \mtx{W} }.
$$
Take the expectation to arrive at
\begin{equation} \label{eqn:psd-lower-max}
\Expect \max\nolimits_i \norm{ \mtx{T}_i } \leq
	\Expect \norm{ \mtx{W} }.
\end{equation}
Average the two bounds~\eqref{eqn:psd-lower-mean} and~\eqref{eqn:psd-lower-max}
to obtain
$$
\Expect \norm{\mtx{W}}
	\geq \frac{1}{2} \big[ \norm{ \Expect \mtx{W} } + \Expect \max\nolimits_i \norm{\mtx{T}_i} \big].
$$
To reach~\eqref{eqn:psd-case-lower}, apply the numerical fact that $2(a + b) \geq \big(\sqrt{a} + \sqrt{b}\big)^{2}$,
valid for all $a, b \geq 0$.

\subsection{Hermitian Case}

The Hermitian case~\eqref{eqn:herm-case-lower} is similar in spirit, but the
details are a little more involved.  Recall that
$$
\mtx{X} := \sum_{i=1}^n \mtx{Y}_i
\quad\text{where $\Expect \mtx{Y}_i = \mtx{0}$.}
$$
First, using the identity~\eqref{eqn:norm-power}, we have
\begin{equation} \label{eqn:herm-lower-var}
\left( \Expect \normsq{ \mtx{X} } \right)^{1/2}
	= \big( \Expect \norm{ \smash{\mtx{X}^2} } \big)^{1/2}
	\geq \norm{ \Expect \mtx{X}^2 }^{1/2}.
\end{equation}
The second relation is Jensen's inequality~\eqref{eqn:jensen}.

To obtain the other part of our lower bound, we use the lower bound 
from the symmetrization result, Fact~\ref{fact:symmetrization}:
$$
\Expect \normsq{ \mtx{X} }
	= \Expect \normsq{ \sum_{i=1}^n \mtx{Y}_i }
	\geq \frac{1}{4} \Expect \normsq{ \sum_{i=1}^n \eps_i \mtx{Y}_i }
$$
where $\{\eps_i\}$ is an independent family of Rademacher random variables,
independent from $\{\mtx{Y}_i\}$.  Now, we condition on the choice of $\{\mtx{Y}_i\}$,
and we compute the partial expectation with respect to the $\eps_i$.
Let $I$ be the minimum value of the index $i$ where $\max\nolimits_i \normsq{\mtx{Y}_i}$
is achieved.  By Jensen's inequality~\eqref{eqn:jensen}, applied conditionally,
$$
\Expect_{\vct{\eps}} \normsq{ \sum_{i=1}^n \eps_i \mtx{Y}_i }
	\geq \Expect_{\eps_I} \normsq{ \Expect\left[ \sum_{i=1}^n \eps_i \mtx{Y}_i \, \big\vert\, \eps_I \right] }
	= \Expect_{\eps_I} \normsq{ \eps_I \mtx{Y}_I }
	= \max\nolimits_i \normsq{\mtx{Y}_i}.
$$
Combining the last two displays and taking a square root, we discover that
\begin{equation} \label{eqn:herm-lower-max}
\left( \Expect \normsq{ \mtx{X} } \right)^{1/2}
	\geq \frac{1}{2} \big( \Expect \max\nolimits_i \normsq{\mtx{Y}_i} \big)^{1/2}.
\end{equation}
Average the two bounds~\eqref{eqn:herm-lower-var} and~\eqref{eqn:herm-lower-max}
to conclude that~\eqref{eqn:herm-case-lower} is valid.

\subsection{The Rectangular Case}

The rectangular case~\eqref{eqn:rect-case-lower} follows instantly from the Hermitian case
when we apply~\eqref{eqn:herm-case-lower} to the Hermitian dilation.  Recall that
$$
\mtx{Z} := \sum_{i=1}^n \mtx{S}_i
\quad\text{where $\Expect \mtx{S}_i = \mtx{0}$.}
$$
Define a random matrix $\mtx{X}$ by applying the Hermitian dilation~\eqref{eqn:dilation} to $\mtx{Z}$:
$$
\mtx{X} := \coll{H}(\mtx{Z}) = \sum_{i=1}^n \coll{H}(\mtx{S}_i).
$$
Since $\mtx{X}$ is a sum of independent, centered, random Hermitian matrices,
the bound~\eqref{eqn:herm-case-lower} yields
$$
\left( \Expect \normsq{ \coll{H}(\mtx{Z}) } \right)^{1/2}
	\geq \frac{1}{2} \norm{ \Expect\big[ \coll{H}(\mtx{Z})^2 \big] }
	+ \frac{1}{4} \big( \Expect \max\nolimits_i \normsq{\coll{H}(\mtx{S}_i)} \big)^{1/2}.
$$
Repeating the calculations in Section~\ref{sec:rect-case-upper},
we arrive at the advertised result~\eqref{eqn:rect-case-lower}.

\section{Optimality of Theorem~\ref{thm:main}} \label{sec:examples} 

The lower bounds and upper bounds in Theorem~\ref{thm:main} match,
except for the dimensional factor $C(\vct{d})$.  In this section, we show
by example that neither the lower bounds nor the upper bounds can
be sharpened substantially.  More precisely, the logarithms cannot appear in
the lower bound, and they must appear in the upper bound.  As
a consequence, unless we make further assumptions, Theorem~\ref{thm:main}
cannot be improved except by constant factors and, in one place, by
an iterated logarithm.

\subsection{Upper Bound: Variance Term}
\label{sec:upper-var}

First, let us show that the variance term in the upper bound
in~\eqref{eqn:main-ineqs} must contain a logarithm.
This example is drawn from~\cite[Sec.~6.1.2]{Tro15:Introduction-Matrix}.

For a large parameter $n$, consider the $d \times d$ random matrix
$$
\mtx{Z} := \sum_{i=1}^d \sum_{j=1}^n \frac{1}{\sqrt{n}} \eps_{ij} \mathbf{E}_{ii}
$$
As before, $\{ \eps_{ij} \}$ is an independent family of Rademacher random variables,
and $\mathbf{E}_{ii}$ is a $d \times d$ matrix with a one in the $(i, i)$ position
and zeroes elsewhere.  The variance parameter satisfies
$$
v(\mtx{Z}) = \norm{ \sum_{i=1}^d \sum_{j=1}^n \frac{1}{n} \mathbf{E}_{ii} }
	= \norm{ \Id_d } = 1.
$$
The large deviation parameter satisfies
$$
L^2 = \Expect \max\nolimits_{i,j} \normsq{ \frac{1}{\sqrt{n}} \eps_{ij} \mathbf{E}_{ii} }
	= \frac{1}{n}.
$$
Therefore, the variance term drives the upper bound~\eqref{eqn:main-ineqs}.
For this example, it is easy to estimate the norm directly.  Indeed,
$$
\Expect \normsq{ \mtx{Z} }
	\approx \Expect \normsq{ \sum_{i=1}^d \gamma_i \mathbf{E}_{ii} }
	= \Expect \max_{i=1,\dots, d} \abssq{\smash{\gamma_i}}
	\approx 2 \log d.
$$
Here, $\{\gamma_i\}$ is an independent family of standard normal variables,
and the first approximation follows from the central limit theorem.
The norm of a diagonal matrix is the maximum absolute value of
one of the diagonal entries.  Last, we use the well-known fact
that the expected maximum among $d$ squared standard normal variables
is asymptotic to $2\log d$.  In summary,
$$
\left( \Expect \normsq{\mtx{Z}} \right)^{1/2}
	\approx \sqrt{2 \log d \cdot v(\mtx{X})}.
$$
We conclude that the variance term in the upper bound must carry a logarithm.
Furthermore, it follows that Theorem~\ref{thm:matrix-rademacher} is numerically sharp.

\subsection{Upper Bound: Large-Deviation Term}

Next, we verify that the large-deviation term in the upper bound in~\eqref{eqn:main-ineqs}
must also contain a logarithm, although the bound is slightly suboptimal.
This example is drawn from~\cite[Sec.~6.1.2]{Tro15:Introduction-Matrix}.

For a large parameter $n$, consider the $d \times d$ random matrix
$$
\mtx{Z} := \sum_{i=1}^d \sum_{j=1}^n \big( \delta_{ij} - n^{-1} \big) \cdot \mathbf{E}_{ii}
$$
where $\{ \delta_{ij} \}$ is an independent family of $\textsc{bernoulli}\big(n^{-1}\big)$ random variables.
That is, $\delta_{ij}$ takes only the values zero and one, and its expectation is $n^{-1}$.
The variance parameter for the random matrix is
$$
v(\mtx{Z}) = \norm{ \sum_{i=1}^d \sum_{j=1}^n \Expect \big( \delta_{ij} - n^{-1} \big)^2 \cdot \mathbf{E}_{ii} }
	= \norm{\sum_{i=1}^d \sum_{j=1}^n n^{-1}\big(1 - n^{-1}\big) \cdot \mathbf{E}_{ii} }
	\approx 1. $$
The large deviation parameter is
$$
L^2 = \Expect \max\nolimits_{i,j} \normsq{ \big(\delta_{ij} - n^{-1} \big)\cdot \mathbf{E}_{ii} }
	\approx 1.
$$
Therefore, the large-deviation term drives the upper bound in~\eqref{eqn:main-ineqs}:
$$
\left( \Expect \normsq{\mtx{Z}} \right)^{1/2} \leq \sqrt{4 (1 + 2\lceil \log d \rceil)} + 4(1 + 2 \lceil \log d \rceil).
$$
On the other hand, by direct calculation
$$
\left( \Expect \normsq{\mtx{Z}} \right)^{1/2}
	\approx \left( \Expect \normsq{ \sum_{i=1}^d (Q_i - 1) \cdot \mathbf{E}_{ii} } \right)^{1/2}
	= \left( \Expect \max_{i=1,\dots,d} \abssq{Q_i - 1} \right)^{1/2}
	\approx \mathrm{const} \cdot \frac{\log d}{\log \log d}.
$$
Here, $\{Q_i\}$ is an independent family of $\textsc{poisson}(1)$ random variables,
and the first approximation follows from the Poisson limit of a binomial.  The
second approximation depends on a (messy) calculation for the expected squared
maximum of a family of independent Poisson variables.  We see that the large
deviation term in the upper bound~\eqref{eqn:main-ineqs} cannot be improved,
except by an iterated logarithm factor.

\subsection{Lower Bound: Variance Term}

Next, we argue that there are examples where the variance term in the lower bound from~\eqref{eqn:main-ineqs}
cannot have a logarithmic factor.

Consider a $d \times d$ random matrix of the form
$$
\mtx{Z} := \sum_{i,j=1}^d \eps_{ij} \mathbf{E}_{ij}.
$$
Here, $\{ \eps_{ij} \}$ is an independent family of Rademacher random variables.
The variance parameter satisfies
$$
v(\mtx{Z}) = \max\left\{ \norm{ \sum_{i,j=1}^d \big( \Expect \eps_{ij}^2 \big)\cdot \mathbf{E}_{ij} \mathbf{E}_{ij}^\adj },
	\norm{ \sum_{i,j=1}^d \big( \Expect \eps_{ij}^2 \big)\cdot \mathbf{E}_{ij}^\adj \mathbf{E}_{ij} } \right\}
	= \max\left\{ \norm{ d \cdot \Id_d }, \norm{ d \cdot \Id_d} \right\}
	= d.
$$
The large-deviation parameter is
$$
L^2 = \Expect \max\nolimits_{i,j} \normsq{ \eps_{ij} \mathbf{E}_{ij} } = 1.
$$
Therefore, the variance term controls the lower bound in~\eqref{eqn:main-ineqs}:
$$
\left( \Expect \normsq{\mtx{Z}} \right)^{1/2} \geq \sqrt{c d} + c.
$$
Meanwhile, it can be shown that the norm of the random matrix $\mtx{Z}$ satisfies
$$
\left( \Expect \normsq{\mtx{Z}} \right)^{1/2} \approx \sqrt{2d}.
$$
See the paper~\cite{BV14:Sharp-Nonasymptotic} for an elegant proof of this nontrivial result.
We see that the variance term in the lower bound in~\eqref{eqn:main-ineqs}
cannot have a logarithmic factor.

\subsection{Lower Bound: Large-Deviation Term}

Finally, we produce an example where the large-deviation term in the lower bound from~\eqref{eqn:main-ineqs}
cannot have a logarithmic factor.

Consider a $d \times d$ random matrix of the form
$$
\mtx{Z} := \sum_{i=1}^d P_i \mathbf{E}_{ii}.
$$
Here, $\{P_i\}$ is an independent family of symmetric random variables
whose tails satisfy
$$
\Prob{ \abs{P_i} \geq t } = \begin{cases} t^{-4}, & t \geq 1 \\ 1, & t \leq 1. \end{cases}
$$
The key properties of these variables are that
$$
\Expect P_i^2 = 2
\quad\text{and}\quad
\Expect \max_{i=1,\dots,d} P_i^2 \approx {\rm const} \cdot d^2.
$$
The second expression just describes the asymptotic order of the expected maximum.
We quickly compute that the variance term satisfies
$$
v(\mtx{Z}) = \norm{ \sum_{i=1}^d \big( \Expect P_i^2 \big) \mathbf{E}_{ii} }
	= 2.
$$
Meanwhile, the large-deviation factor satisfies
$$
L^2 = \Expect \max_{i=1,\dots, d} \normsq{ P_i \mathbf{E}_{ii} }
	= \Expect \max_{i=1, \dots, d} \abssq{P_i}
	\approx \mathrm{const} \cdot d^2.
$$
Therefore, the large-deviation term drives the lower bound~\eqref{eqn:main-ineqs}:
$$
\left( \Expect \normsq{\mtx{Z}} \right)^{1/2}
	\gtrapprox \mathrm{const} \cdot d.
$$
On the other hand, by direct calculation,
$$
\left( \Expect \normsq{\mtx{Z}} \right)^{1/2}
	= \left( \Expect \normsq{ \sum_{i=1}^d P_i \mathbf{E}_{ii} } \right)^{1/2}
	= \left( \Expect \max_{i=1,\dots,d} \abssq{ P_i } \right)^{1/2}
	\approx \mathrm{const} \cdot d.
$$
We conclude that the large-deviation term in the lower bound~\eqref{eqn:main-ineqs}
cannot carry a logarithmic factor.

\section*{Acknowledgments}

The author wishes to thank Ryan Lee for a careful reading of the manuscript.
The author gratefully acknowledges support from
ONR award N00014-11-1002 and the Gordon \& Betty Moore Foundation.

\bibliographystyle{myalpha}
\newcommand{\etalchar}[1]{$^{#1}$}

\end{document}